\documentclass[a4paper,12pt]{amsart}

\usepackage[utf8]{inputenc}

\usepackage{amsmath, amsthm, amsfonts, amssymb}
\usepackage[foot]{amsaddr}
\usepackage{stmaryrd}	
\usepackage{graphicx}
\usepackage{todonotes}
\usepackage{xcolor}

\usepackage{tikz}
\usepackage{todonotes}
\usepackage{hyperref} 
\usepackage[mathlines, pagewise]{lineno} 
\usepackage{pb-diagram}
\usepackage{float}
\usepackage{enumitem}

\newtheorem{theorem}{Theorem}
\newtheorem{lemma}[theorem]{Lemma}
\newtheorem{corollary}[theorem]{Corollary}
\newtheorem{prop}[theorem]{Proposition}

\theoremstyle{definition}

\theoremstyle{remark}
\newtheorem{remark}[theorem]{Remark}

\newcommand{\norm}[1]{\left\Vert#1\right\Vert}
\newcommand{\abs}[1]{\left|#1\right|}
\newcommand{\ipr}[2]{\left\langle#1,#2\right\rangle}
\newcommand{\bk}{\mathbf{k}}
\newcommand{\bx}{\mathbf{x}}
\newcommand{\by}{\mathbf{y}}
\newcommand{\lin}{\ensuremath{\mathrm{lin}}}
\newcommand{\mix}{\ensuremath{\mathrm{mix}}}
\DeclareMathOperator{\Id}{Id}

\begin{document}

\title{Sampling recovery in $L_2$ and other norms}

\author{David Krieg}

\author{Kateryna Pozharska}

\author{Mario Ullrich}

 \author{Tino Ullrich}

\thanks{DK was supported by the Austrian Science Fund (FWF) M~3212-N. 
For the purpose of open access, the author has applied a CC BY public copyright licence to any Author Accepted Manuscript version arising from this submission. KP would like to acknowledge support by the Philipp Schwartz Fellowship of the Alexander von Humboldt Foundation and the budget program of the NAS of Ukraine ``Support for the development of priority areas of research'' (KPKVK 6541230). MU is supported by the Austrian Federal Ministry of Education, Science and Research via the Austrian Research Promotion Agency (FFG) through the project FO999921407 (HDcode) funded by the European Union via NextGenerationEU.
TU and KP are supported by the German Research Foundation (DFG) with grant Ul403/4-1\,.}


\date{}



\begin{abstract}
We study the recovery of functions in various 
norms, 
including $L_p$ with $1\le p\le\infty$, 
based on function evaluations. 
We obtain worst case error bounds for general classes of functions
in terms of the best $L_2$-approximation from a given nested sequence of subspaces 
and
the Christoffel function of these subspaces. 
In the case $p=\infty$, 
our results 
imply that linear sampling algorithms are optimal up to a constant factor
for many reproducing kernel Hilbert spaces.
\end{abstract}

\maketitle


\section{Introduction}

Let $F$ be a class of complex-valued functions on a domain $D$
and let $G$ be some semi-normed space that contains $F$.
The $n$-th linear sampling number of the class $F$ in $G$ is defined by
\begin{align*}
g_n^{\lin}(F,G) \,:=\, 
\inf_{\substack{x_1,\dots,x_n\in {D}\\ \varphi_1,\dots,\varphi_n\in G}}\, 
\sup_{f\in F}\, 
\Big\|f - \sum_{i=1}^n f(x_i)\, \varphi_i\Big\|_G.
\end{align*}
This is the \emph{minimal worst case error} that 
can be achieved with 
a linear sampling recovery algorithm
based on at most $n$ function values 
if the error is measured in~$G$. 

The sampling numbers are often difficult to analyze
and it is therefore of interest to compare them to other approximation quantities
which are either easier to handle or which are already well-studied.
This is the case for the so-called
\emph{linear $n$-width} or \emph{$n$-th approximation number}. 
Those numbers are defined as
\begin{equation}\label{appnu}
a_n(F, G) \,:=\,
\inf_{\substack{T\colon F \to G\\ {\rm rank}(T) \,\le\, n}}\, 
\sup_{f\in F}\, 
\big\|f - Tf \big\|_{G},
\end{equation}
i.e., they represent the error of best approximation by a {linear} operator of rank at most $n$.
Alternatively, they may be described as
the minimal worst case error 
that can be achieved with arbitrary linear algorithms
that use up to $n$ arbitrary linear measurements. 

Recently, there has been much progress in the study of sampling numbers in the case $G=L_2$.
For example,
if $F=B_H$ is the unit ball of a separable reproducing kernel Hilbert space $H$
that is compactly embedded into $L_2$,
a sharp bound on the sampling numbers was obtained in~\cite{DKU}. 
It was proven there that there is a universal constant $c\in\mathbb{N}$ such that 
\begin{equation*}
g_{cn}^\lin(B_H,L_2) \,\le\, \sqrt{ \frac{1}{n} \sum_{k\ge n} a_k(B_H,L_2)^2 }
\end{equation*}
for all $n\in\mathbb{N}$.
For a much more general family of classes $F$ 
(like compact subsets of the space of bounded functions),
the similar relation
\begin{equation}\label{eq:DKU-general}
g_{cn}^\lin(F,L_2) \,\le\, \bigg( \frac{1}{n} \sum_{k\ge n} a_k(F,L_2)^p \bigg)^{1/p}
\end{equation}
was shown for all $p<2$ with a constant $c\in\mathbb{N}$ determined by $p$.
In particular, these results show that 
\[
g_{n}^\lin(F,L_2) \;\asymp\; a_n(F,L_2)
\]
if $a_n(F,L_2)$ decays polynomially with rate greater than $1/2$, 
and hence, that (linear) algorithms based on function values
are asymptotically as powerful as arbitrary linear algorithms 
in this case. 
See also \cite{KU1,KU2,KWW,NSU} for forerunners of the above results,  
\cite{HNV,HKNV,KV} for corresponding lower bounds, 
and \cite{T20} for an upper bound in terms of Kolmogorov numbers in $L_\infty$, 
which is particularly useful if the numbers $(a_n)$ 
do not decay fast enough to apply the previous results, see \cite{TU1}. 

The main aim of the present paper is to depart from the assumption $G=L_2$ and provide bounds for the sampling numbers in general norms.
We summarize our main contributions.
\begin{itemize}[leftmargin=*]
    \item   \textbf{Section~\ref{sec:L2}:}  We start with a slight improvement of the upper bound~\eqref{eq:DKU-general} from~\cite{DKU} for the $L_2$-case. 
    We obtain the formula
    \[
g_{cm}^{\rm lin}(F,L_2) 
\;\le\; \frac{1}{\sqrt{m}}\sum_{k > m} \frac{ a_k(F,L_2)}{\sqrt{k}}
\]
 with a universal constant $c>0$,     see Corollary~\ref{coro:L2}
    and the discussion there.
    \medskip
    
    \item \textbf{Section~\ref{sec:general}:}  Most importantly, we give upper bounds on the sampling numbers $g_{n}^\lin(F,G)$ for error norms $G\ne L_2$, including 
    $G=L_p$. 
    As above, those upper bounds are in terms of the ``simple'' quantity $a_n(F,L_2)$.
    In order to deviate from the assumption $G=L_2$ we use a ``lifting'' approach similar to the one in \cite{PU} where the authors consider approximation in $G=L_\infty$.
    For example,
    whenever the optimal subspaces for $L_2$-approximation are ``nice enough" (we will define later what this means), we obtain the formula
 (see Corollary~\ref{coro:Linf})
 \[
    g_{cm}^{\rm lin}(F,L_\infty) 
    \;\le\; C \, \sum_{k > m} \frac{ a_k(F,L_2)}{\sqrt{k}}.
    \]
   
        \medskip

    \item \textbf{Section~\ref{sec:Hilbert}:} If $F=B_H$ is the unit ball of a reproducing kernel Hilbert space, 
    the bounds from Section~\ref{sec:general} can be improved.
    This is especially true if $G=L_\infty$ and the Hilbert space satisfies a certain additional assumption. Then we obtain a direct comparison of the numbers $g_{n}^\lin(B_H,L_\infty)$ and $a_n(B_H,L_\infty)$, namely,
    \[
     g_{cm}^{\rm lin}(B_H,L_\infty)
     \le C\, a_m(B_H,L_\infty),
    \]
    see Theorem~\ref{thm:H}.
    This implies that (linear) sampling algorithms are as powerful as arbitrary (linear) algorithms for uniform approximation
    on such Hilbert spaces.
          \medskip

    \item \textbf{Section~\ref{sec:ex}:} We apply our results to a variety of examples: 
    spaces of periodic functions (Section~\ref{subsec:per}),
    Sobolev spaces on manifolds (Section~\ref{subsec:mani}),
    and an example based on Legendre polynomials (Section~\ref{ex:unbounded}).
    The algorithms that our approach is based on are (weighted or unweighted) least-squares algorithm.
         {In particular,}
    we compare our results for the least-squares algorithm (LS) with available bounds for the Smolyak algorithm in the case that $F$ is the unit ball of a mixed order Sobolev 
    {$\mathbf{W}^r_p$}
    {or Nikol'skii $\mathbf{H}^r_p$ spaces.
    The outcome is described by Figure~\ref{fig1}; see Section~\ref{subsec:per} for details.}
\end{itemize}
          
\definecolor{both-opt}{rgb}{1,0,1}

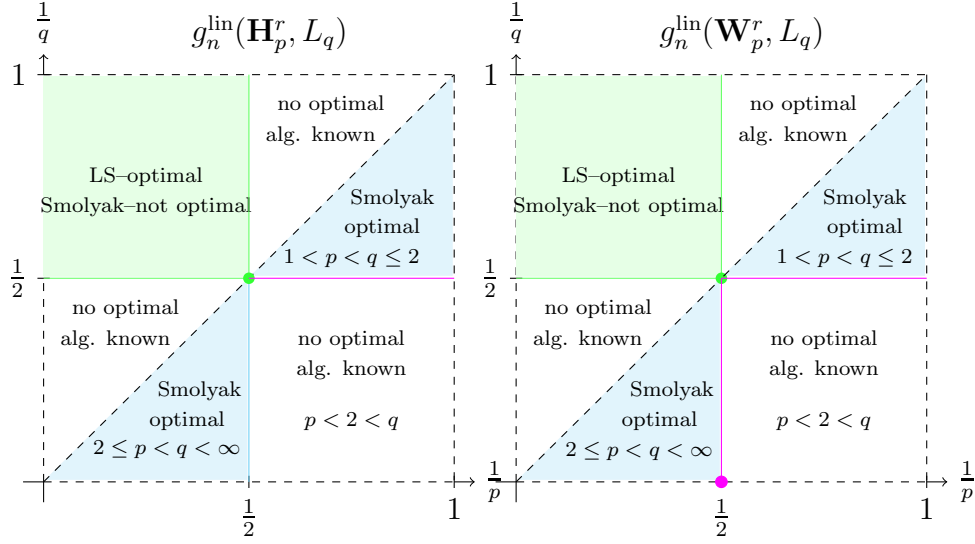
\begin{figure}[h]
	\begin{minipage}{0.45\textwidth}
		\begin{tikzpicture}[scale=2.4]
	
		\draw (-0.1,0.0) -- (0,0 );
		\draw[->] (2,0) -- (2.1,0.0) node[right] {$\frac{1}{p}$};
		\draw[->] (0,2) -- (0.0,2.1) node[above] {$\frac{1}{q}$};
		
		\draw[dashed](0,0) -- (0.0,2);
		\draw (0,-0.1) -- (0,0);
		\draw (1.0,0.03) -- (1.0,-0.03) node [below] {$\frac{1}{2}$};
		\draw (0.03,1) -- (-0.03,1) node [left] {$\frac{1}{2}$};

		\node at (1.1,2.2) {$g_n^{\lin}(\mathbf{H}^r_p, L_q)$};
		
		\draw[green!60] (0,2) -- (1,2);
  \draw[dashed] (1,2) -- (2,2);

				\filldraw[green!10] (0,1.01) -- (1,1.01) -- (1,1.99) -- (0,1.99);
    \fill[green!80] (1,2)  circle[radius=0.8pt];
	
	\node at (.5,1.5) {\tiny  LS--optimal};
	\node at (0.5,1.35) {\tiny Smolyak--not optimal};
	\draw[cyan!50] (1,0) -- (1,1);

		\draw[green!60] (1,1) -- (1,2);
		\draw[green!60] (0,1) -- (1,1);	
		\draw[dashed] (2,2) -- (2,0);

		\filldraw[cyan!10] (0.01,0.01) -- (0.99,0.99) -- (0.99,0.01);
			\filldraw[cyan!10] (1.01,1.015) -- (1.99,1.015) -- (1.99,1.99);
			
		\node at (.76,0.45) {\tiny  Smolyak};
			\node at (.7,0.3) {\tiny  optimal};
		
				\node at (.6,0.16) {\tiny  $2 \leq p<q<\infty$ };
				
	\node at (1.7,1.39) {\tiny  Smolyak};
\node at (1.5,1.1) {\tiny 	$1<p<q \leq 2$};
	\node at (1.65,1.25) {\tiny 	optimal};

		\draw[dashed] (0,0) -- (2,2);
		\draw[dashed] (0,0) -- (2,0);

\fill[green!80] (1,1)  circle[radius=0.8pt];
	
		\draw (2,0.03) -- (2,-0.03) node [below] {$1$};
		\draw (0.03,2) -- (-0.03,2) node [left] {$1$};
		
		\node at (1.5,0.7) {\tiny  no optimal};
		\node at (1.5,0.55) {\tiny alg. known };

		\node at (1.5,0.3) {\tiny $p<2<q$};

			\node at (1.4,1.85) {\tiny  no optimal};
		\node at (1.35,1.7) {\tiny alg. known };

			\node at (0.4,0.85) {\tiny  no optimal};
		\node at (0.35,0.7) {\tiny alg. known };

  \draw[both-opt!100] (1,1) -- (2,1);

		\end{tikzpicture}
	\end{minipage}
    \hskip3mm
		\begin{minipage}{0.45\textwidth}
		\begin{tikzpicture}[scale=2.4]
		\draw (-0.1,0.0) -- (0,0 );
		\draw[->] (2,0) -- (2.1,0.0) node[right] {$\frac{1}{p}$};
		\draw[->] (0,2) -- (0.0,2.1) node[above] {$\frac{1}{q}$};
		
		\draw[dashed](0,0) -- (0.0,2);
		\draw (0,-0.1) -- (0,0);
		\draw (1.0,0.03) -- (1.0,-0.03) node [below] {$\frac{1}{2}$};
		\draw (0.03,1) -- (-0.03,1) node [left] {$\frac{1}{2}$};

		\node at (1.1,2.2) {$g_n^{\lin}(\mathbf{W}^r_p, L_q)$};
		
	
	\draw[green!60] (0,2) -- (1,2);
  \draw[dashed] (1,2) -- (2,2);

		\filldraw[green!10] (0,1.01) -- (1,1.01) -- (1,1.99) -- (0,1.99);
	
			\node at (.5,1.5) {\tiny  LS--optimal};
			\node at (0.5,1.35) {\tiny Smolyak--not optimal};
	
			\draw[green!60] (1,1) -- (1,2);
		\draw[green!60] (0,1) -- (1,1);	
 
		\draw[dashed] (2,2) -- (2,0);
	
		\filldraw[cyan!10] (0.01,0.01) -- (0.989,0.99) -- (0.989,0.01);
			\filldraw[cyan!10] (1.01,1.015) -- (1.989,1.015) -- (1.989,1.99);
			\fill[green!80] (1,2)  circle[radius=0.8pt];
			
			\node at (.76,0.45) {\tiny  Smolyak};
					\node at (.6,0.15) {\tiny  $2 \leq p<q<\infty$ };
			\node at (.7,0.3) {\tiny  optimal};

			\node at (1.7,1.39) {\tiny  Smolyak};
\node at (1.6,1.1) {\tiny 	$1<p<q \leq 2$};
	\node at (1.7,1.25) {\tiny 	optimal};

\fill[green!80] (1,1)  circle[radius=0.8pt];
		
		\draw[dashed] (0,0) -- (0.99,0.99);
		\draw[dashed] (1.01,1.01) -- (2,2);
		\draw[dashed] (0,0) -- (2,0);

		\draw (2,0.03) -- (2,-0.03) node [below] {$1$};
		\draw (0.03,2) -- (-0.03,2) node [left] {$1$};

		\node at (1.5,0.7) {\tiny  no optimal};
		\node at (1.5,0.55) {\tiny alg. known };

		\node at (1.5,0.3) {\tiny $p<2<q$};

			\node at (1.4,1.85) {\tiny  no optimal};
		\node at (1.35,1.7) {\tiny alg. known};

			\node at (0.4,0.85) {\tiny  no optimal };
		\node at (0.35,0.7) {\tiny alg. known};

    \draw[both-opt!100] (1,0) -- (1,1);
    \draw[both-opt!100] (1,1) -- (2,1);	
    \filldraw[draw=
   both-opt!100, fill=both-opt!100] (1,0)  circle[radius=0.8pt];

		\end{tikzpicture}
	\end{minipage}
	\caption{ Parameter region with sharp orders. \\
\footnotesize (Green area: Least squares is optimal, Smolyak  is not optimal. \\ 
Blue area: Smolyak is optimal, {known LS bounds are not optimal. }\\
Magenta lines: Both algorithms are optimal.)}
	\label{fig1}
\end{figure}

\subsection{Notation} \
For a measure $\mu$ on a set $D$ {and} $1\le p<\infty$, 
we denote by $L_p{:=}L_p(\mu)$  
the space of complex-valued functions that are 
$p$-integrable w.r.t.\ $\mu$, 
and by $L_\infty$ the space of bounded functions on $D$.
The norms in the respective spaces are given by 
$\Vert f \Vert_p \ := \left(\int_{D} |f|^p\,{\rm d}\mu\right)^{1/p}$ 
and $\Vert f \Vert_\infty \,:=\, 
\sup_{x\in D}\; |f(x)|$. 
{That is, $L_\infty$ does \emph{not} denote the space of essentially bounded functions, but instead the space of bounded functions.}
Moreover, 
if $D$ is a {compact} 
topological space, then $C(D)$ is the space of all continuous functions on $D$ {equipped with the max-norm}.  \\
If $H$ is a Hilbert space we denote with $\ipr{f}{g}_H$ the inner product of $f,g \in H$. By $\log$ we denote the logarithm with base 2 and $a_+ := \max \{a, 0 \}$.
For two sequences
$(a_n)_{n=1}^{\infty},(b_n)_{n=1}^{\infty}\subset \mathbb{R}$ we write 
$a_n \lesssim b_n$ if there exists a constant $c>0$,
such that $a_n \leq cb_n$ for all {but finitely many}~$n$. 
We write $a_n \asymp b_n$ if $a_n \lesssim b_n$ and $b_n \lesssim a_n$.
If the constant $c$ depends, 
 {e.g.}, only on the two parameters $\alpha$ and $\beta$, 
then we indicate it by $\lesssim_{\alpha, \beta}$ and $\asymp_{\alpha, \beta}$.
Without indication it may depend on all involved parameters, except for $n$.

\section{Approximation in $L_2$}\label{sec:L2}

The aim of the present paper is to identify properties of function classes $F$ that allow for the existence of good linear sampling recovery 
algorithms for 
approximation in a given seminorm. 

We work under the following assumption, which is an extension of Assumption~A from~\cite{DKU}.

\hypertarget{assum}{}

\noindent\textbf{Assumption B.} \ Assume that:
\vspace*{-3pt}
\begin{itemize}
    \item[(B.1)] $F$ is a separable metric space of complex-valued functions on a set $D$ such that function evaluation $f\mapsto f(x)$ is, for each $x\in D$, continuous on $F$.
    \item[(B.2)] $\mu$ is a measure on $D$ and $F$ is continuously embedded into $L_2(\mu)$.
   \item[(B.3)] 
    $G$ is a seminormed space  {of complex-valued functions on~$D$} which contains $F$, and $G\cap L_2$ is complete w.r.t.~the natural seminorm $\|\cdot\|_*:=\|\cdot\|_G+\|\cdot\|_2$. 
   If two functions from $G$ are equal $\mu$-almost everywhere, then their seminorm in $G$ is the same.
   \item[(B.4)] $(V_n)_{n=1}^\infty$ is a 
   nested sequence of subspaces of $G \cap L_2$ of 
   dimension
    {$\dim V_n = n$ for $n\in \mathbb{N}$.}
    \end{itemize}

All the assumptions are satisfied, for instance, if
\begin{itemize}
    \item $D$ is a compact topological space, 
    \item $\mu$ is a {finite} Borel
    measure with support $D$, 
        \item $F$ is a compact subset of the space $C(D)$ of continuous functions (equipped with the metric induced by the maximum norm), 
    \item and $G$ is either $C(D)$ or the semi-normed space $L_p(\mu)$ of $p$-integrable functions on $D$, where $1\le p \le \infty$.
\end{itemize}
 But in principle, also infinite measures $\mu$ and other seminorms like 
 Sobolev norms 
 are allowed, and it is not needed that $F\subset C(D)$ to obtain results for $L_p$ or even uniform approximation. See Section~\ref{ex:unbounded} for an example that goes beyond the aforementioned special case.

Although the main goal of this paper is to study the sampling numbers in norms different from $L_2$, our first contribution is a slight improvement of the best known bound in the case $G=L_2$.
The main assumption here is that there exists a sequence of  {nested} subspaces 
$V_n$ 
that are ``good'' for $L_2$-approximation. We compare the linear sampling numbers with the error of best $L_2$-approximation within the spaces $V_n$, i.e., with

\begin{equation}\label{eq:an}
\alpha_n \,:=\, \alpha(V_n,F, L_2
) \,:=\, \sup_{f \in F} \Vert f - P_{V_n} f  \Vert_2
\;=\; \sup_{f \in F}\, \inf_{g\in V_n}\Vert f - g  \Vert_2, 
\end{equation}
where $P_{V_n}$ is the orthogonal projection onto $V_n$ in $L_2(\mu)$.

Based on the ideas from~\cite{DKU}, we obtain that there are optimal sampling algorithms 
whenever the numbers $\alpha_n$ decay fast enough. 

\begin{prop}\label{prop:L2}
There is an absolute constant $c>0$ such that the following holds.
Let $D$, $\mu$ and $F$ satisfy \hyperlink{assum}{Assumption~B} with $G=L_2$. 
Moreover, let $(\alpha_k)_{k=1}^\infty$ be 
as in \eqref{eq:an}.
Then, for every $m\in\mathbb{N}$, there 
are $x_1,\dots,x_n\in D$ and $\varphi_1,\dots,\varphi_n\in V_m$, where $n\le cm$,
such that the algorithm $A_m
{f}
=\sum_{i=1}^n f(x_i) \varphi_i$ satisfies
\begin{equation*}
\sup_{f\in F}\, \Vert f - A_m f \Vert_2 
\;\le\; \frac{70}{\sqrt{m}}\sum_{k > \lfloor m/4 \rfloor} \frac{ \alpha_k}{\sqrt{k}}.
\end{equation*}
\end{prop}

\begin{remark}[The algorithm] \label{rem:alg}
{The algorithm in Proposition~\ref{prop:L2}
is in fact a \emph{weighted least squares method}. It
is the same algorithm as used
in~\cite{DKU} for $L_2$-ap\-proxi\-mation.
Unfortunately,
the result from \cite{DKU} is not constructive.
However, as already observed in~\cite{KU1} we obtain the same error bound with high probability if, instead of the $cm$ unknown sampling points, we use a weighted least squares method based on $c m \log m$ i.i.d.\ random sampling points with respect to a specific density (given there), 
see also~\cite{KU2,MU,U2020}, 
or~\cite{SU23} for a survey.
The density is determined by the spaces $(V_m)$ and the numbers $(\alpha_k)$.
The random construction can be implemented whenever the density is ``nice enough".
This is the case, for instance, if the spaces $V_m$ are of the form $V_m = \operatorname{span}\{b_1,\hdots,b_m\}$ with a uniformly bounded orthonormal system $(b_j)_{j\in\mathbb{N}}$ in $L_2$,
or also in the situation of Theorem~\ref{thm:H} (for Hilbert spaces) below.
Then the points can be taken uniformly distributed on the domain and the algorithm is a \emph{plain} (non-weighted) \emph{least squares method}.}
\\
{In order to reduce the $cm\log m$ points to $cm$ points,}
a first subsampling approach 
was considered in~\cite{NSU}. 
This uses a (finite-dimensional) \emph{frame subsampling} 
that is based on 
the non-constructive solution to the Kadison-Singer problem. 
In~\cite{DKU}, this was extended to an infinite-dimensional version, 
leading to an optimal bound. 
It is open if this can be done constructively, 
see e.g.~\cite{BSU} for first results. 
\end{remark}

The main idea 
of the proof is to construct an appropriate ``intermediate'' 
\emph{reproducing kernel Hilbert space} {(RKHS)} $H$ and then apply the bounds available for Hilbert spaces, 
i.e., Theorem~23 from~\cite{DKU}.

\begin{proof}
{We follow the lines of \cite[Proof of Theorem~3]{DKU} and only choose the singular values $\sigma_k$ differently. We copy the proof for convenience. 

Clearly, we may assume that $\alpha_k= \alpha(V_k,F, L_2)$ {defined by \eqref{eq:an}} 
is finite for $k\ge k_0 := \lfloor m/4\rfloor$ and that $(\alpha_k/\sqrt k)_{k\ge k_0}$ is summable. Otherwise, the statement is trivial.
This summability (together with the monotonicity of $(\alpha_k)_k$) also implies that $(\alpha_k)_{k\ge k_0}$ is square summable.

First, we choose functions $(b_k)_{k\in \mathbb{N}_0}$ such that they are orthonormal in $L_2$ and such that $V_m = \operatorname{span}\{b_k\colon k<m\}$.
We let $P_m$ be the $L_2$-orthogonal projection onto $V_m$.
We call $$\hat f(k) := \langle f, b_k\rangle_{L_2}$$
the $k$th Fourier coefficient of $f$.
Since $\alpha_k \to 0$ for $k\to \infty$,
we have $P_k f \to f$ in $L_2$ and so
\[
 f = \sum_{k\in\mathbb{N}_0} \hat f(k) b_k
\]
for all $f\in F$ with convergence in $L_2$,
as well as
\[
 \Vert f - P_k f \Vert_2^2 \,=\, \sum_{r\ge k} |\hat f(r)|^2 \,\le\, \alpha_k^2
\]
for all $r\in \mathbb{N}$,
and hence also
\begin{equation}\label{eq:RM}
 \sum_{k\ge 1} k\, |\hat f(k)|^2 
 \,=\, \sum_{n\ge 0} \sum_{k>n} |\hat f(k)|^2
 \, < \infty.
\end{equation}
Define $(\sigma_k)_{k\in\mathbb{N}_0}$ to be a non-increasing sequence with
\begin{equation*}
\sigma_k^2 \,:=\, 
\frac{\alpha_{\lceil k/2\rceil}}{\sqrt{k}}
\qquad \text{if } \lceil k/2\rceil \ge k_0.
\end{equation*}
Then we have $(\sigma_k)\in\ell_2$.

We fix a countable dense subset $F_0$ of $F$.
We want to define a 
RKHS on a set $D_0 \subset D$,
with $\mu(D\setminus D_0)=0$, which has the orthonormal basis $(\sigma_k b_k){_{k\in \mathbb{N}_0}}$
and contains the set $F_0$.
Such a Hilbert space will have the reproducing kernel 
\[
 K(x,y) \,=\, \sum_{k\in\mathbb{N}_0} \sigma_k^2 b_k(x) \overline{b_k(y)}.
\]
To find a suitable set $D_0$, we first note that 
\begin{equation}\label{eq:traceofH}
 \int_D K(x,x)\,d\mu(x) \,=\, \sum_{k\in\mathbb{N}_0} \sigma_k^2 < \infty
\end{equation}
and thus $K(x,x)$ is finite for all $x\in D\setminus E$
with a null set $E\subset D$.
Moreover, for all $f\in F_0$, via the relation \eqref{eq:RM},
the Rademacher-Menchov Theorem implies that 
the Fourier series of $f$ at $x$ converges to $f(x)$ 
for all $x\in D\setminus E_f$ with a null set $E_f\subset D$.
We put 
$D_0 := D \setminus E_0$, where
$
 E_0 := E \cup \bigcup_{f\in F_0} E_f
$
is a null set.
Then
for all $x\in D_0$ and $f\in F_0$, we have
\[
 K(x,x) \,< \infty
\quad \text{and} \quad
 f(x) \,=\, \sum_{k\in \mathbb{N}_0} \hat f(k)\, b_k (x).
\]

We now define the space $H$ as the set of all square-integrable functions $f\colon D_0 \to \mathbb{C}$
which are pointwise represented by their Fourier series
$\sum_{{k\in \mathbb{N}_0}} \hat f (k) \, b_k$
and which satisfy
\[
 \Vert f \Vert_H^2 \,:=\, \sum_{k\in\mathbb{N}_0} \frac{|\hat f(k)|^2}{\sigma_k^2} < \infty.
\]
Then $H$ is a separable reproducing kernel Hilbert space on $D_0$
since
\[
|f(x)|^2\leq K(x,x)\|f\|_H^2\quad\text{ for all }x\in D_0\text{ and }f\in H,
\]
and $(\sigma_k b_k)_{k\in\mathbb{N}_0}$ is an orthonormal basis of $H$.
The reproducing kernel is $K$, which
has finite trace by \eqref{eq:traceofH}.

We now show that $F_0$ (with functions restricted to $D_0$) is a subset of $H$.
Recall that any $f\in F_0$ is pointwise represented by its Fourier series. Moreover,
note that $\alpha(V_k,F_0, L_2)=\alpha(V_k,F, L_2)=\alpha_k$  for all $k$.
So}
\[ 
\begin{split}     
	\Vert f - P_m f &\Vert^2_H 
	\,=\, \sum_{k > m} \sigma_k^{-2} \abs{\ipr{f}{b_k}_2}^2 
	\,\le\, \sum_{ \ell  =0}^\infty \frac{\sqrt{2^{\ell+1}m}}{\alpha_{2^{\ell}m}} 
	\sum_{k= 2^{\ell}m +1}^{2^{\ell+1}m}  \abs{\ipr{f}{b_k}_2}^2\\
	&\le\, \sum_{ \ell  =0}^\infty \sqrt{2^{\ell+1}m} \cdot \alpha_{2^{\ell}m}
	\,\le\, \sum_{ \ell  =0}^\infty \sqrt{2^{\ell+1}m} \cdot \frac{1}{2^{\ell-1}m} \sum_{k=\lfloor 2^{\ell-1}m\rfloor+1}^{2^\ell m} \alpha_k\\
	&\le\, 2\sqrt2 \sum_{ \ell  =0}^\infty \sum_{k=\lfloor 2^{\ell-1}m\rfloor+1}^{2^\ell m} \frac{ \alpha_k}{\sqrt{k}}
	\,=\, 2\sqrt2 \sum_{k> \lfloor m/2\rfloor} \frac{ \alpha_k}{\sqrt{k}} \,< \infty.
	\end{split}
\]
{This implies also that $f\in H$ for all $f\in F_0$.

We now apply \cite[Theorem~23]{DKU}
to the newly constructed Hilbert space $H$ to
find $n\le 43200\,m$ and a linear algorithm $A_m$ of the form
\[
 A_m(f) \,=\, \sum_{i=1}^n f(x_i) \varphi_i, \quad x_i \in D_0, \ \varphi_i \in V_m,
\]
such that
\begin{equation*}
 \Vert f - A_m f \Vert_{L_2(D_0,\mu)}^2 
 \,\le\, 433 \,
\max\bigg\{\sigma_m^2, \, \frac{1}{m}\sum_{k\geq m}\sigma_k^2 \, \bigg\}
\, \Vert f - P_m f \Vert_{H}^2
\end{equation*}
for all $f\in H$ and thus, for all $f\in F_0$.
Clearly, in the last inequality, $D_0$ can be replaced with $D$.
If we now insert the upper bound for $\Vert f - P_m f\Vert_H^2$ and the estimate}
$$
\sigma_m^2 \leq \frac{1}{m/2} \sum_{k =  \lfloor m/2\rfloor+1 }^m \sigma_k^2 \leq \frac{2}{m} \sum_{k >  \lfloor m/2\rfloor  } \sigma_k^2
\,\leq \frac{4}{m} \sum_{k > \lfloor m/4 \rfloor} \frac{\alpha_k}{\sqrt k},
$$
we get
\[
 \Vert f - A_m f \Vert_{L_2}^2 
 \,\le\, \frac{\sqrt{4\cdot 2\sqrt2 \cdot 433}}{\sqrt{m}}\sum_{k > \lfloor m/4 \rfloor} \frac{ \alpha_k}{\sqrt{k}}
\]
for {all $f\in F_0$.
Since $F_0$ is dense in $F$ and
both ${\rm id} \colon F \to L_2$ and $A_m \colon F \to L_2$
are continuous,
this bound holds for all $f\in F$. 
This concludes the proof.
}
\end{proof}

If we choose the sequence of nested subspaces $V_m$ accordingly,
we immediately get an upper bound on the linear sampling numbers in terms of the approximation numbers.

\begin{corollary}\label{coro:L2}
There is an absolute constant $c\in \mathbb{N}$ such that,
for all $D$, $\mu$ and $F$ that satisfy 
Assumptions \hyperlink{assum}{(B.1)} and \hyperlink{assum}{(B.2)},
and for all $m\in\mathbb{N}$, we have 
\[
g_{cm}^{\rm lin}(F,L_2) 
\;\le\; \frac{1}{\sqrt{m}}\sum_{k > m} \frac{ a_k(F,L_2)}{\sqrt{k}}.
\]
\end{corollary}

\begin{proof}
By \cite[Lemma~3]{KU2}, there is a nested sequence of subspaces  {$V_k$} of $L_2$ such that {for $\alpha_k$ from \eqref{eq:an} it holds}
$
{\alpha_k}
 \le 2 a_{\lfloor k/4 \rfloor}(F,L_2)
$.
{We note that 
\hyperlink{assum}
{Assumption~B} is satisfied for $G=L_2$ and this choice of spaces $V_n$. So we can apply}
Proposition~\ref{prop:L2} to get
\[
g_{cm}^{\rm lin}(F,L_2) 
\;\le\; \frac{560}{\sqrt{m}}\sum_{k > \lfloor m/16 \rfloor} \frac{ a_k(F,L_2)}{\sqrt{k}}.
\]
with $c$ as in Proposition~\ref{prop:L2}.
Finally, we replace $c$ with $\lceil 560^2c \rceil$.
\end{proof}

\sloppy
We note that the assumptions 
\hyperlink{assum}
{(B.1)} and 
\hyperlink{assum}
{(B.2)} together equal the Assumption~A from~\cite{DKU}.
Corollary~\ref{coro:L2} is a slight improvement of \cite[Thm.~3]{DKU} and \cite[Thm.~27]{DKU}. 
{This is most obvious if the sequence $(k^{-1/2} a_k)_{k\in\mathbb{N}}$ is ``just barely summable". 
For example, if $a_k{(F,L_2)}= k^{-1/2} (\log k)^{-1} (\log\log k)^{-t}$ for some $t>1$,
then the previous upper bounds are infinite, whereas the new upper bound is small}.


\bigskip

\section{Approximation in general norms}\label{sec:general}

For approximation in general norms, 
the idea is to lift the previous results from $L_2$-approximation to $G$-approximation, using 
the quantity 

\begin{equation}\label{eq:Lambda}
\Lambda_n \;:=\;
\Lambda(V_n,G) \;:=\; 
 \sup\limits_{f \in V_n, \, f\neq 0} \frac{\|f\|_G}{\|f\|_2}.  
\end{equation}
Depending on the choice of $V_n$, {$G$ and $L_2(\mu)$,}
the $\Lambda_n$
correspond to the inverse of the Christoffel function, or the best constant in Bernstein or Nikol'skii inequalities.
Its appearance is no surprise as 
it is 
an important quantity in approximation theory, 
especially for $G=L_\infty$ and $V_n$ a space of polynomials, see~\cite{Freud,Gr19,Nevai}.  
It has already 
turned out to be instrumental in 
proving results in the uniform norm, see e.g.~\cite{KWW,PU}.

\smallskip
\goodbreak

The general procedure of lifting results from $L_2$-approximation to $G$-approximation will be explained 
{below}.
The outcome is the following result for linear sampling recovery 
in general norms.

\smallskip

\begin{theorem}\label{thm:general}
Let 
\hyperlink{assum}
{Assumption~B} hold and
let $(\alpha_k)_{k=1}^\infty$ and $(\Lambda_k)_{k=1}^\infty$ be defined as in \eqref{eq:an} and \eqref{eq:Lambda}. 
Moreover, let $c$ and $A_m$ 
from Proposition~\ref{prop:L2}
be the universal constant 
and the linear algorithm
$A_m \colon  F \to V_m$, that uses  $n\leq c m $ function evaluations, 
respectively. \\
Then, for all $m\in\mathbb{N}$,
we have 
$$
\sup_{f\in F}\, \Vert f - A_m f  \Vert_{G} \,\le
 {2} \sum_{k > \lfloor m/4 \rfloor} \frac{\alpha_k \Lambda_{4k}}{k} 
	+ \frac{70\Lambda_m}{\sqrt{m}}\sum_{k > \lfloor m/4 \rfloor} \frac{ \alpha_k}{\sqrt{k}}.
$$
\end{theorem}

\smallskip
\goodbreak

Taking into account that 
$
\sum_{k>m} k^{-a} (\log k)^b
 \,\lesssim_{a}\, m^{-a+1} (\log m)^b
 $
 whenever $a>0$ and $b\in\mathbb{R}$, Theorem~\ref{thm:general} yields
the following corollary on the order of convergence of the sampling numbers.

\begin{corollary}\label{coro:general}   
Let 
\hyperlink{assum}
{Assumption~B} hold and let 
$a,b,\gamma$ and $\delta$ be real parameters with $a > \max\{b, 1/2\}$.
If
\[
\Lambda(V_n,G)  
\;\lesssim\; n^b (\log n)^{\delta}
\quad \text{ and } \quad 
\alpha(V_n,F, L_2)
\,\lesssim\, n^{-a} (\log n)^\gamma
\]
then
\[
g_n^{\lin} (F, G) \;\lesssim\; n^{-a+b} (\log n)^{\gamma+\delta}.
\]
\end{corollary}

\medskip

In Section~\ref{sec:ex} we discuss examples that show that 
this bound is often sharp up to constants.

For finite measures and uniform approximation, the optimal behavior of the Christoffel function is that $\Lambda(V_m,L_\infty) \asymp \sqrt{m}$, see Remark~\ref{rem:minimal-Chris}.
In this case, 
our upper bound simplifies as follows.

\begin{corollary}\label{coro:Linf}
 There is a universal constant $c\in\mathbb{N}$ such that,
 whenever the assumptions of Theorem~\ref{thm:general} hold together with $\Lambda(V_m,L_\infty) \le C m^{1/2}$, then
\[
 g_{cm}^\lin( F , L_\infty) \,\le\, 74 C \sum_{k > m} \frac{ \alpha_k}{\sqrt{k}}.
\]
\end{corollary}

\begin{proof}
 We use Theorem~\ref{thm:general} with $4m$ instead of $m$.
\end{proof}

\begin{remark}[{The case $G=L_\infty$}]
\label{rem:minimal-Chris} 
{In this paper, $L_\infty$ denotes the space of bounded functions equipped with the strong sup-norm. Our general
\hyperlink{assum}
{Assumption~B} would also admit to choose $G$ as the space of essentially bounded functions with the weak sup-norm, but we are more interested in pointwise estimates. We choose the (somewhat unusual) notation $L_\infty$ for the space of bounded functions for convenience, in order to have a unified notation for all $1\le p \le \infty$.
If $V_n$ is an $n$-dimensional subspace of $L_\infty$, the numbers $\Lambda(V_n,L_\infty)$ can be expressed in terms of an arbitrary $L_2$-}orthonormal basis 
$\{ b_k\}_{k=1}^{n}$ 
of $V_n$. 
We have
\begin{equation}\label{gamma2}
\Lambda(V_n,L_\infty) \,=\, 
\Big\Vert\sum_{k=1}^{n} \left| b_k \right|^2\Big\Vert_{\infty}^{1/2}.
\end{equation}
If the measure $\mu$ is finite, then
\begin{equation*}
    \Lambda(V_n,L_\infty) \,\ge\,  \bigg( \frac{1}{\mu(D)} \int_D \sum_{k=1}^{n} \left| b_k \right|^2 {\rm d}\mu \bigg)^{1/2}
    \,=\, \sqrt{\frac{n}{\mu(D)}}.
\end{equation*}
That is, the numbers $\Lambda(V_n,L_\infty)$ grow at least like $\sqrt{n}$.
This optimal behavior of $\Lambda_n$ appears
for spaces spanned by
trigonometric monomials, Chebyshev polynomials and spherical harmonics if $\mu$ is their corresponding orthogonality measure. In addition, it is met for certain Walsh and 
 (Haar) wavelet spaces (however, not in general).
In other cases, {$\Lambda(V_n,L_\infty)$} might be larger. 
We will discuss some examples in Section~\ref{sec:ex}.
\end{remark}

\begin{remark}[{The case $G=L_p$}]
{For} $2\le p < \infty$,
{and in the setting of Remark~\ref{rem:minimal-Chris},} an easy computation shows
\[
\Lambda(V_n,L_p) \,\le\, \Lambda(V_n,L_\infty)^{1-2/p}
 \,=\,
 \Big\Vert\sum_{k=1}^{n} \left| b_k \right|^2\Big\Vert_{\infty}^{1/2-1/p}.
\]
\end{remark}

There are 
numerous results in the literature 
on uniform and $L_p$-approximation, often for specific classes $F$ 
and explicit algorithms, see Section~\ref{sec:ex}. 
The advantage of Theorem~\ref{thm:general} 
is its generality and simplicity. 
{Since the used algorithm (see Remark \ref{rem:alg} for details) is specified 
by the knowledge of $F$ and $(V_n)_{n=1}^\infty$ only, 
Theorem~\ref{thm:general} shows the existence of 
an algorithm that is \emph{universal} in $p$.}

\bigskip

{To explain the lifting that is needed for the proof of Theorem~\ref{thm:general}, }
let $F$, $G$ and $(V_n)$ that satisfy 
\hyperlink{assum}
{Assumption~B} be fixed, and  
recall 
the shorthand notations 
\begin{equation*}
\Lambda_n \;
{=}
\;\Lambda(V_n,G) \;
{=}
\; \sup\limits_{f \in V_n, \, f\neq 0} \frac{\|f\|_G}{\|f\|_2} 
\end{equation*}
and 
\begin{equation*}
\alpha_n \,
{=}
\, \alpha(V_n,F, L_2) \,
{=}
\, \sup_{f \in F} \Vert f - P_{V_n} f  \Vert_2
\;=\; \sup_{f \in F}\, \inf_{g\in V_n}\Vert f - g  \Vert_2, 
\end{equation*}
where $P_{V_n}$ is the orthogonal projection onto $V_n$ in $L_2(\mu)$. 
We also choose an orthonormal system $\{b_k\}_{k=1}^\infty$
such that $b_1,\dots,b_n$
form a basis of $V_n$ for all $n\in\mathbb{N}$.

{
We now show how to lift $L_2$-error bounds
to error bounds in $G$ using the spectral function $\Lambda_n$.
}

\begin{lemma}\label{lem:lift}
 Let \hyperlink{assum}{Assumption~B} hold and let 
 $(\alpha_k)_{k=1}^\infty$ and $(\Lambda_k)_{k=1}^\infty$ be defined as in~\eqref{eq:an} and \eqref{eq:Lambda}, respectively.
 For any $m\in\mathbb{N}$, any mapping $A\colon F \to V_m$ and all $f\in F$,  we have
 \begin{equation*}
  \Vert f - Af\Vert_{G}\,\le\, 
  2 \sum_{k > \lfloor m/4 \rfloor 
    } \frac{ \alpha_k \, \Lambda_{4k}}{k} 
  \,+\, \Lambda_m \cdot \Vert f-Af\Vert_2.
 \end{equation*}
 \end{lemma}

 \begin{proof}
Note that there is nothing to prove if the right hand side of the inequality 
is infinite
 so we may assume that it is finite.
We first observe
\begin{align*}
 \Vert f - A f  \Vert_{G}
&\le\, \Vert f - P_{V_m} f \Vert_{G} + \Vert P_{V_m} f - A f \Vert_{G}\\
&\le\, \Vert f - P_{V_m} f \Vert_{G} + \Lambda_m \cdot \Vert P_{V_m}f - A f \Vert_{2}\\
&\le\, \Vert f - P_{V_m} f \Vert_{G} + \Lambda_m \cdot \Vert f - A f \Vert_{2}.
\end{align*}

In order to estimate $\Vert f - P_{V_m} f \Vert_G$,
we first 
show that $P_{V_m} f$ converges in $G\cap L_2$.
We 
 {define}
$I_\ell = \{ k \in \mathbb{N} \colon 2^{\ell}<  k \leq 2^{\ell+1}\}$ for $\ell\in\mathbb{N}_0$. Let $m\in I_s$ and $n\in I_t$ with $s\le t$. Then

\begin{align*}
\Vert P_{V_n} &f - P_{V_m} f \Vert_G  \,\le\,
\sum_{\ell=s}^t
\bigg\Vert\sum\limits_{ k\in I_\ell\colon m<k\le n} \ipr{f}{b_k}_2 b_k \bigg\Vert_G
\\ 
&\le\,
\sum_{\ell=s}^t
\bigg\Vert\sum\limits_{ k\in I_\ell} \ipr{f}{b_k}_2 b_k \bigg\Vert_2
  \cdot \Lambda_{2^{\ell+1}}
 \,\leq\, \sum_{\ell \geq s}  \alpha_{2^\ell} \Lambda_{2^{\ell+1}}.
\end{align*}
 Further we use the fact that the sequence $(\alpha_k)_{k=1}^\infty$
 is non-increasing, and therefore
\begin{equation*}
\alpha_{2^\ell} \leq   \frac{1}{2^{\ell-1}} \sum\limits_{j\in I_{\ell-1}}  {\alpha_j,}
	\end{equation*} 
and hence, since the sequence $(\Lambda_k)_{k=1}^\infty$
is non-decreasing,
\[
\begin{split}
\Vert P_{V_n} f - &P_{V_m} f \Vert_G
 \,\leq\, 2\sum_{\ell \geq s}  \sum\limits_{ j\in I_{\ell-1}}
\frac{\alpha_j \Lambda_{2^{\ell+1}}}{2^{\ell} } 
\\
& \leq 2\sum_{\ell \geq s}  \sum\limits_{ j\in I_{\ell-1}}
\frac{\alpha_j \Lambda_{4j}}{j}   
\,\le\, 2 \sum_{j > \lfloor m/4 \rfloor} \frac{ \alpha_j \Lambda_{4j}}{j},
\end{split}
\]
which is convergent due to the assumption.
This means that $(P_{V_m} f)$ is a Cauchy sequence in $G$.
Since we also know that $P_{V_m} f$ converges to $f$ in $L_2$, 
we obtain that  $(P_{V_m} f)$ is a Cauchy sequence in $G\cap L_2$ and converges to some $g\in G\cap L_2$ by Assumption (B.3). In particular, $f=g$ almost everywhere.
 Thus, using again Assumption (B.3), we obtain 
\[
 \Vert f - P_{V_m} f \Vert_G \,=\,
 \Vert g - P_{V_m} f \Vert_G \,\le\, 
 2 \sum_{j > \lfloor m/4 \rfloor} \frac{ \alpha_j \Lambda_{4j}}{j},
\]
which yields the statement.
 \end{proof}

If we insert the algorithm from Proposition~\ref{prop:L2}
into Lemma~\ref{lem:lift}, we immediately get our main result for approximation in general seminorms, i.e., Theorem~\ref{thm:general} 
{above.}


\bigskip

\section{Improvements for Hilbert spaces}\label{sec:Hilbert}

Our results can be further improved if we consider 
{uniform approximation ({i.e., $G=L_\infty$}) in}
unit balls of 
Hilbert spaces. 
This also leads to some interesting results on 
the power of linear algorithms based on function values 
compared to \emph{arbitrary algorithms}. 
Let us recall that in this case, i.e., 
if the error is measured with $\|\cdot\|_G=\|\cdot\|_\infty$, we approximate 
with respect to the (strong) sup-norm. \\

For a comparison of the power of different algorithms or information, 
one usually takes the Gelfand widths of the class $F$ as the ultimate benchmark. The \emph{Gelfand $n$-width} of $F$ in $G$ is defined by
\begin{align*}
c_n(F, G) \,:=\,
\inf_{\substack{\psi\colon \mathbb{C}^n\to G\\ 
N\colon F\to\mathbb{C}^n \; \text{ linear}}}\; 
\sup_{f\in F}\, 
\big\|f - \psi\circ N(f) \big\|_{G}
\end{align*}
and measures the worst case error of the optimal (possibly non-linear) algorithm using $n$ pieces of arbitrary linear information.
{We clearly have the relations 
$g_n^\lin \ge a_n \ge c_n$ 
whenever function evaluation 
is a continuous linear functional on $F$ (for every $x\in D$). 
Moreover, 
it is a classical result that 
$a_n(F,L_\infty)=c_n(F,L_\infty)$ 
if $F$ is symmetric and convex, 
but they are different in many other situations, 
see \cite[Section~4.2.2]{NW1}.
}

Unfortunately,
Theorem~\ref{thm:general} does not allow for a direct comparison of $g_n^\lin(F,G)$ and $c_n(F,G)$. 
This is clearly a desirable goal
since it corresponds to the question
whether, 
for a fixed ``budget'' ($n$) and under the given 
assumptions on the ``problem instances'' ($f\in F$),
other algorithms may perform better 
than linear algorithms based on function values. 
The subsequent
results show
that the answer is \emph{no} for a certain familiy of classes~$F$.
Namely, by techniques very similar to the ones used for proving Theorem~\ref{thm:general}, we obtain the following 
improvement in the case of 
RKHS 
which satisfy an additional boundedness assumption. 
In this case, we obtain 
that linear sampling algorithms are optimal among all 
(possibly non-linear) algorithms; 
a result that has been observed independently in~\cite[Thm.~2.2]{GW23}. 
{We use the following notation.}

\begin{remark}[Notation] 
{Here and in the following,
if $F$ is a normed space,
we abuse notation
by writing $g_n^\lin(F,G)$ when we actually mean
$g_n^\lin(B_F,G)$ where $B_F$ is the unit ball of $F$. The same is done for the numbers $a_n$ and $c_n$.}
\end{remark}

\medskip

\begin{theorem}\label{thm:H}
There are
absolute constants 
{$c,C \in\mathbb{N}$}
such that the following holds.
Let $\mu$ be 
a measure
on a set $D$ 
and
$H\subset L_2(\mu)$ be a 
reproducing kernel Hilbert space 
with kernel
$$
K(x,y) = \sum_{k=1}^{\infty} \sigma_k^2\, b_k(x)\, \overline{b_k(y)}, 
\qquad x,y\in D, 
$$
where $(\sigma_k)\in \ell_2$ is a non-increasing sequence 
and $\{b_k\}_{k=1}^\infty$
is an orthonormal system in
$L_2$ 
such that there is a constant $B>0$ with
\begin{equation}\label{cond_boundedness}
    \Big\Vert \frac{1}{n}\sum_{k=1}^{n} \left| b_k \right|^2\Big\Vert_{\infty} 
\,\le\, B
\end{equation}
for all $n\in\mathbb{N}$. Under these conditions we have $H \subset L_\infty$ and 
\begin{equation}
\label{eq:Hilbert1}
g_{{cn}}^\lin( H , L_{\infty})^{2} \,\le\, 
 B {\cdot C \cdot}\sum\limits_{k > n}\sigma_k^2. 
\end{equation}
In particular, if $\mu$ is a finite measure, then
\begin{equation}
\label{eq:Hilbert2}
 g_{{cn}}^\lin( H , L_{\infty}) \,\le\, \sqrt{
  {B \cdot C} \cdot \mu(D)}\cdot c_n(H,L_\infty).
\end{equation}
\end{theorem}

\smallskip

\begin{remark}[Tractability]
Theorem~\ref{thm:H} is particularly useful if we 
consider approximation for multivariate functions, 
e.g.~on $D=[0,1]^d$, and are interested 
in the dependence of the complexity on the dimension $d$,
that is, in the tractability of the recovery problem. \\
For example, for uniform approximation of periodic functions, 
the theorem implies that the curse of dimensionality holds for approximation 
based on function values (sometimes called \emph{standard information}) 
if and only if it holds for linear information. 
Details can be found in~\cite{GW23}.
The corresponding statement for $L_2$-recovery is not true, 
see, e.g., \cite{HKNV}, 
but similar results can be obtained if $c_n(H,L_2)$ decays exponentially 
with $n$, see~\cite{KSUW}.
\end{remark}

\begin{remark}[$s$-numbers]
The numbers 
$a_n$ and $c_n$ 
are well-studied objects in approximation theory and 
the geometry of Banach spaces, and are therefore well-known 
for many classes $F$. 
In fact, they fall under the general category of \emph{$s$-numbers} 
which allows the use of some fundamental techniques. 
We refer to~\cite{Diestel-book,Pietsch-s} for a general treatment of this, 
and to~\cite{DTU18,KKLT,Tem18} for more approximation theoretic 
background. \\
In contrast, the numbers $g_n^\lin$ 
do not (yet) 
fit into such a general framework, and their dependence on $n$ is 
often still unknown. 
It is therefore of interest, and one of the 
major concerns of 
\emph{information-based complexity}, 
to determine the relations of all the above widths depending on 
properties of the classes $F$, 
see e.g.~\cite{NW1,NW3,TWW88}. 
\end{remark}

Theorem~\ref{thm:H} 
represents a special case of a more general result for reproducing kernel Hilbert spaces 
where we do not need the uniform boundedness of the basis, see our Theorem~\ref{thm:H2} below, 
and \cite[Thm.~2.1]{GW23}.
These results further improve on the previous results of \cite{KWW08} in the Hilbert space setting, 
see~\cite[Sect.~5]{PU} for a discussion.

Another general bound in this context is due to Novak, see~\cite[Prop.~1.2.5]{Novak} or \cite[Thm.~29.7]{NW3}, 
which states that 
\begin{equation}\label{eq:Novak}
g_n^\lin(F,L_\infty) 
 \,\le\, (1+n)\,c_n(F,L_\infty)
\end{equation}
whenever $F$ is a Banach space. 
In fact, \eqref{eq:Novak} holds with the (smaller) Kolmogorov widths in $L_\infty$ on the right hand side, see \cite{Novak,NW3} for details. 
The factor $(1+n)$ cannot be improved:
For each $n$, there is even a Hilbert space  where \eqref{eq:Novak} is an equality, 
see~\cite[Example~29.8]{NW3}. \\

In the paper \cite{KPUU23_2} it is shown that the factor $(1+n)$ can be improved to $\sqrt{n}$ when allowing for a constant oversampling factor, i.e., $c\cdot n$ sampling points 
instead of $n$. 

Let us emphasize  that the result in Theorem \ref{thm:H} applies to several important 
cases of orthonormal systems, which are listed in Remark~\ref{rem:minimal-Chris}. 
Moreover, 
we do not require any 
decay of the Gelfand width $c_n$.
In this 
sense, the comparison statement is even stronger than the recently obtained 
result for $L_2$-approximation {discussed above}. 

{In particular, }
there are many cases where 
the minimal error $c_n(F,L_2)$ of non-linear algorithms decays much faster  {than the minimal error $a_n(F, L_2)$ of linear ones}.
For example, for the class $W^s_1=W^s_1([0,1])$, $s\ge2$, 
i.e., 
the unit ball of the 
univariate 
 {Sobolev}
space 
on the interval equipped with the norm 
$\|f\|_{W^s_p}:=\|f\|_p+\|f^{(s)}\|_p$, 
we have 
$a_n(W^s_1,L_2)\asymp \sqrt{n}\, c_{n}(W^s_1,L_2)\asymp n^{-s+1/2}$, 
see e.g.~\cite[Theorem~VII.1.1]{Pinkus85} and the references therein.

\goodbreak 

In contrast, Theorem~\ref{thm:H} above shows that
the optimality of linear algorithms based on function values
is true for uniform approximation
without assumption on the main rate $\alpha$,  
or restriction to linear algorithms, 
at least for certain Hilbert spaces.


{
For proving our main result, we need some prerequisites.  
Afterwards, we will basically repeat the steps from above in this setting.}

\subsection{Preliminaries}
Let
$H$ be a RKHS of complex-valued functions defined on 
a set $D$ with a 
bounded kernel $K\colon D\times D \rightarrow \mathbb{C}$  of the form 
\begin{equation}\label{kernel}
K(x,y) = \sum_{k=1}^{\infty}\sigma_k^2\, b_k(x) \overline{b_k(y)},
\end{equation}
with pointwise convergence on $D \times D$, where $\sigma_1\ge \sigma_2\ge \hdots \ge 0$
 and $\{b_k\}_{k=1}^\infty$ is orthonormal in $L_2(\mu)$ w.r.t.\ the measure $\mu$.

For the embedding operator $\Id\colon H \rightarrow L_2(\mu)$, 
the sequence $( \sigma_k)_{k=1}^{\infty}$ is the sequence of its singular numbers. 
Further we put $e_k(x) := \sigma_k  b_k (x) $, $k=1, 2, \dots$.
 The crucial property for functions $f$ from the RKHS $H$ is the identity
 \begin{displaymath}
 	f(x) = \ipr{f}{K(\cdot,x)}_H
 \end{displaymath}
 for all $x \in D$. It ensures that point evaluations are continuous functionals on $H$. 
 
 In the framework of this paper, the boundedness of the kernel $K$ is assumed, i.e.,
 \begin{equation} \label{CK000}
 	\|K\|_{\infty} := \sup\limits_{x \in D} \sqrt{K(x,x)} < \infty.
 \end{equation}
 The condition (\ref{CK000}) implies that $H$ is continuously embedded into $L_\infty$, i.e.,
 \begin{displaymath}
 	\|f\|_\infty \leq  \|K\|_{\infty}\cdot\|f\|_{H}\,.
 \end{displaymath}

 Taking into account \cite[Lemma 2.6]{Steinwart_Scovel2012},
 we have that the set $\{e_k\}_{k=1}^{\infty}$ 
is an orthonormal basis in $H$. 
 Hence, for every $f\in H$, it
  {holds that}
 \begin{equation*}
 	f = \sum_{k=1}^{\infty} \ipr{f}{e_k}_{H} e_k,
 \end{equation*}
 where $\ipr{f}{e_k}_{H}$, $k=1, 2, \dots$, are the Fourier coefficients of $f$ with respect to the system $\{ e_k\}_{k=1}^{\infty}$. Let us further 
  {define}
 \begin{equation*}
 	P_mf := \sum\limits_{k=1}^{m} \ipr{f}{e_k}_{H} e_k
 \end{equation*}
 the projection onto the space $\operatorname{span}\{e_1,...,e_{m}\}$,
which we denote by $V_m$.
 Due to Parseval's identity, the norm of the space $H$ takes the following form
 \begin{displaymath}
 	\|f\|^2_{H} =   \sum_{k=1}^{\infty} | \ipr{f}{e_k}_{H} |^2 .
 \end{displaymath}
We also define
 \[
  K_m(x,y) := \sum_{k>m} \sigma_k^2 b_k(x) \overline{b_k(y)},
 \]
 i.e., the reproducing kernel of the orthogonal complement of $V_m$ in $H$.
 {In this section, we only consider uniform appproximation and the corresponding Christoffel function of the subspaces $V_m$, i.e.,
 \begin{equation}\label{eq:Christoffel-H}
     \Lambda_m \,=\, \Lambda(V_m,L_\infty)
     \,=\, \sup\limits_{f \in V_m, \, f\neq 0} \frac{\|f\|_\infty}{\|f\|_2}.
 \end{equation}}

 \subsection{{The result}} 
 
Using the Christoffel function \eqref{eq:Christoffel-H},
 we can {again} lift $L_2$-error bounds
 for $V_m$-valued algorithms
 to $L_\infty$-error bounds.
 
 \begin{lemma}\label{lem:liftH}
 For any mapping $A\colon H \to V_m$ and all $f\in H$, we have
 \begin{equation}\label{eq:liftH}
  \Vert f - Af\Vert_\infty\,\le\, 
  \Vert K_m \Vert_\infty
  \cdot \Vert f-P_m f\Vert_{H}
  \,+\, \Lambda_m \cdot \Vert f-Af\Vert_2,
 \end{equation}
 where
 \[
 \Vert K_m \Vert_\infty^2
 \,\le\, 2 \sum\limits_{k> \lfloor m/4 \rfloor}\frac{\Lambda_{4k}^2\sigma_k^2}{k}.
 \]
 \end{lemma}

  \begin{proof}
By the triangle inequality,
 \begin{equation}\label{Th1_splitting}
	\Vert f - A f \Vert_\infty
\,\le\, \Vert f - P_m f \Vert_\infty + \Vert P_m f - A f \Vert_\infty. 
\end{equation}
To estimate the first term in (\ref{Th1_splitting}), we 
use that $(\Id - P_m)^2 = \Id - P_m$ and that the operator $\Id - P_m$ is self-adjoint, 
and get, for all $x\in D$,
\begin{equation*}
	\begin{split}	\abs{(f - P_m f)(x)}^2  &
		= \abs{ \ipr{(\Id - P_m) f}{K(\cdot, x)}_H }^2\\
		&= \abs{ \ipr{(\Id - P_m)^2 f}{K(\cdot, x)}_H }^2
\\
&= \abs{ \ipr{(\Id - P_m) f}{(\Id - P_m) K(\cdot, x)}_H }^2\\
&= \abs{ \ipr{f - P_m f}{K_m(\cdot, x)}_H }^2\\
		&\leq
		\|f - P_m f\|_{H}^2 \cdot K_m(x,x).
	\end{split}
\end{equation*}	
We turn to the second term in \eqref{Th1_splitting}.
Since $A$ maps to $V_m$, we have $A f = P_m (A f)$, and hence
\begin{align*}
\Vert P_m f - &A f \Vert_\infty  = \Vert P_m (f - A f) \Vert_\infty \\
&\leq \Lambda_m \cdot \Vert  P_m( f - A f) \Vert_{2}
\leq \Lambda_m \cdot \Vert  f - A f \Vert_{2}.
\end{align*}
This gives \eqref{eq:liftH}. It remains to bound $K_m(x,x)$.
Using that the sequence $( \sigma_k)_{k=1}^{\infty}$ is non-increasing, 
we obtain for $2^\ell \leq k < 2^{\ell+1}$ that 
$$
\sigma_k^2 \leq  \frac{1}{2^{\ell-1}} \sum\limits_{ 2^{\ell-1}< j \le 2^\ell}\sigma_j^2,
$$
which implies  {by (\ref{gamma2}) that}
\begin{equation*}
K_m(x,x)
		\le \sum_{\ell \geq \lfloor \log m \rfloor} \sum_{2^\ell \leq k < 2^{\ell+1}}
		\sigma_k^2 \abs{b_k(x)}^2
 \leq   \sum_{\ell \geq \lfloor \log m \rfloor}
		\frac{ \Lambda_{2^{\ell+1}}^2 }{2^{\ell-1}} \sum_{ 2^{\ell-1}< j \le 2^\ell}\sigma_j^2 \, .
\end{equation*}	
Hence, in view of the relations $2^{\ell+1} \leq 4j$ and $1/2^{\ell-1} \le 2/j$, that are true for all $2^{\ell-1}< j \le 2^\ell$, we get
\begin{equation*}
K_m(x,x)
	\,\leq \sum_{\ell  \geq \lfloor \log m \rfloor} \sum_{ 2^{\ell-1}< j \le 2^\ell}  \frac{2 \Lambda_{4j}^2}{j} \sigma_j^2
	 \,\leq\, 2 \sum\limits_{k > \lfloor m/4 \rfloor}\frac{\Lambda_{4k}^2\sigma_k^2}{k}.
\end{equation*}
 \end{proof}

According to \cite[Thm.~23]{DKU}, 
there is a linear algorithm 
\begin{equation*}
A_m\colon H \to V_m, 
\quad A_m(f) = \varphi(f(x_1),\hdots,f(x_{n}))
\end{equation*}
using $n\le c m$ samples, 
satisfying the $L_2$-error bound
\begin{align}\label{eq:DKU}
\Vert f - A_m f \Vert_2^2 \,&\le\, 
433 \,\max\bigg\{\sigma_{m+1}^2, \frac{1}{m} \sum_{k  > m } \sigma_k^2\bigg\}\cdot  \Vert f - P_m f \Vert_H^2\\ \nonumber
&\le\, \frac{866}{m} \sum_{k  > \lfloor m/2 \rfloor } \sigma_k^2 \cdot  \Vert f - P_m f \Vert_H^2
\end{align}
for all $f\in H$, where $c>0$ is a universal constant.
If we plug this algorithm into Lemma~\ref{lem:liftH},
we immediately get the following result.
Recall that this result has recently also been observed in \cite{GW23}.

 \begin{theorem}\label{thm:H2}
There is an
absolute constant
$c>0$,
such that the following holds.
Let $(D,\Sigma,\mu)$ be a measure space
and
$H\subset L_2(\mu)$ be a 
RKHS with bounded kernel $K\colon D\times D\to \mathbb{C}$ of the form \eqref{kernel} and $\Lambda_m$ be defined {as in \eqref{eq:Christoffel-H}},
where $V_m=\operatorname{span}\{b_1, \dots, b_m\}$.
Then it holds for all $m\in\mathbb{N}$ that
\[
 g^\lin_{cm}( H, L_{\infty}) \,\le\,
 \sqrt{\frac{866 \,\Lambda_m^2}{m}\sum_{k > \lfloor m/2 \rfloor}\sigma_k^2} \,+\,
 \sqrt{2 \sum_{k > \lfloor m/4 \rfloor} \frac{\Lambda_{4k}^2\sigma_k^2}{k}} \, .
\]
\end{theorem}
 
This estimate simplifies a lot
in the case that $\Lambda_m^2$ grows linearly with $m$,
as we stated in Theorem~\ref{thm:H}.
In this case, our upper bounds are sharp up to constants
due to a matching lower bound on the Gelfand width in $L_\infty$. 
In fact, it has been shown
in \cite[Thm.~3]{Osipenko_Parfenov1995}
that
\begin{equation}\label{cobos}
	\sum_{k>m} \sigma_k^2  \,\leq\,
	\mu(D) \cdot c_m(H , L_{\infty})^2, 
\end{equation}
see also
\cite[Lemma~3.3]{CKS}, \cite[Proposition~2.7]{Kunsch2018}, and \cite[Thm.~4]{KWW08}.

\begin{proof}[Proof of Theorem~\ref{thm:H}]
We apply Theorem~\ref{thm:H2} with $4n$ instead of $m$ 
{and note that $\Lambda_n^2 \le B n$ to obtain \eqref{eq:Hilbert1}.
The bound \eqref{eq:Hilbert2} follows from \eqref{eq:Hilbert1} due to \eqref{cobos}.}
{Respectively, the constant $C$ in the upper bound can be taken as $C=33\geq \sqrt{866}+\sqrt{4\cdot 2}$.}
\end{proof}

\begin{remark}
The assumption on the basis in Theorem~\ref{thm:H} is clearly fulfilled 
if the basis is uniformly bounded, 
i.e., if 
\[
 \sup_{k\in \mathbb{N}}\, \sup_{x \in D} \vert b_k(x)  \vert \,\leq\, B.
\]
For example, this is the case for the trigonometric system on the unit cube with $B=1$.
In this case,
we can improve the constants
if we use the equality $\Vert K_m\Vert_\infty^2 = \sum_{k> m} \sigma_k^2$
in Lemma~\ref{lem:liftH}.
\end{remark}


\bigskip
\goodbreak

\addtocontents{toc}{\protect\setcounter{tocdepth}{2}}
\section{Examples} 
\label{sec:ex}

We now apply our results to a wide range of Hilbert as well as non-Hilbert classes of functions.

\subsection{Spaces of periodic functions} \, 
\label{subsec:per}

\sloppy
  Let $\mathbb{T}= [0, 2 \pi]$ be a torus where the endpoints of the interval are identified  {with each other}. By $\mathbb{T}^d$ we denote a $d$-dimensional torus and equip it with the normalized Lebesgue measure $(2\pi)^{-d} {\rm d} \bx$. For functions 
  $f\in L_1:=L_{1}(\mathbb{T}^d)$ we define
  the Fourier coefficients 
$	f_\bk= (2\pi)^{-d} \int_{\mathbb{T}^d} f(\bx) e^{-i \bk \cdot \bx} {\rm d} \bx$, $\bk\in \mathbb{Z}^d$,
 w.r.t.\ the trigonometric system
$\{ e^{i \bk \cdot \bx } \colon\allowbreak \bk \in \mathbb{Z}^d \}$
which forms a $1$-bounded orthonormal basis of $L_{2}$, 
i.e., for 
$f\in L_{2}$
we have 
$f(\bx) = \sum_{\bk\in \mathbb{Z}^d} f_\bk  e^{i \bk \cdot \bx}$ 
in the sense of convergence in $L_2$.
 {Note that since the system of trigonometric monomials satisfies condition (\ref{cond_boundedness}), it suffices to use a plain least squares algorithm for the recovery, see also Remark~\ref{rem:alg}.}

\subsubsection{Hilbert spaces on the $d$-torus}
Let us consider the space $H^w$ of 
integrable 
functions $f\in L_1(\mathbb{T}^d)$ 
such that
\begin{displaymath}
	\| f \|^2_{H^w} :=	 \sum_{\bk\in \mathbb{Z}^d} (w(\bk))^2 | f_\bk |^2  
    < \infty,
\end{displaymath}
where $w(\bk)>0$, $\bk\in \mathbb{Z}^d$, are certain weights which satisfy the condition 
\begin{equation}\label{sum_1/w^2}
	\sum_{\bk\in \mathbb{Z}^d} (w(\bk))^{-2} < \infty \,
\end{equation}
and $f_\bk$ are the Fourier coefficients w.r.t.\ the trigonometric system
$\left\{ e^{i \bk \cdot \bx }\colon \bk \in \mathbb{Z}^d \right\} =: \{b_k \colon k\in \mathbb{N}\}$.
{It was shown in~\cite[Thm.~3.1]{CKS} that the condition (\ref{sum_1/w^2}) on the weights
$w(\bk)$ is necessary and sufficient for the existence of a compact embedding $H^w \hookrightarrow L_{\infty}$.
The space $H^w$ is a RKHS with kernel
\begin{equation*}
	K_w(\bx, \by) = \sum_{ \bk\in \mathbb{Z}^d } \frac{ e^{i \bk \cdot (\bx - \by)}  }{ \left( w(\bk) \right)^2 }\,,
\end{equation*}
which is bounded under the condition (\ref{sum_1/w^2}).}
We get the following.

\begin{corollary}\label{SC_general_weights}
There are absolute constants
$C,c>0$,
such that for the RKHSs $H^w$ introduced above it holds for all $m\in \mathbb{N}$ that 
	$$
	g_{cm}^{\lin}( H^w, L_{\infty}) \leq  C \sqrt{ \sum\limits_{k>m}	\sigma_k^2}
	\,\le\, 
C\cdot c_{m}(H^w, L_{\infty}), 
	$$
where $\{\sigma_k \colon k \in \mathbb{N} \}$ is the non-increasing rearrangement of the sequence  $\{ (w(\bk))^{-1}\colon \bk \in \mathbb{Z}^d \}$.
\end{corollary}

\begin{proof}[Proof of Corollary~\ref{SC_general_weights}]
The singular numbers $\{\sigma_k \colon k \in \mathbb{N} \}$ of the operator 
$\Id \colon H^w \to L_{2}$ coincide with 
$(w(\bk))^{-1}$, which are rearranged in 
a non-increasing order.
We apply Theorem~\ref{thm:H} with
$\{ b_k \colon k \in \mathbb{N} \} = \{ e^{i \bk \cdot \bx }  \colon \bk \in \mathbb{Z}^d \}$,
where
$\Vert b_k  \Vert_\infty \leq 1$
for all $k\in\mathbb{N}$.
\end{proof}

	The result of Corollary \ref{SC_general_weights} improves the estimate from \cite[Thm. 6.2]{PU} where the authors have an additional $\sqrt{\log m}$-factor on the right hand side. 

We also mention here the earlier results \cite[Theorem~4.11]{Kam19},  
\cite{Kam_Vol19}, \cite{Zeng_Kritzer_Hickernell2009}.

In particular, if $w(\bk) =  \prod_{j=1}^d (1+  |k_j|)^{s}$\,, $\bk=(k_1,\dots,k_d)\in \mathbb{Z}^d $,  $s > 1/2$, then for the unit ball of the periodic Sobolev space 
\begin{equation*}
	H^{s}_{\mix} = \Big\{ f\in L_2 \colon \, 
	\sum_{ \bk\in \mathbb{Z}^d }  | f_\bk |^2  \prod_{j=1}^d (1+  |k_j|)^{2s}  \le 1 \Big\},
\end{equation*}
see, e.g., \cite{KSU}, 
we obtain the following.

\begin{theorem}\label{Th_Sobolev_type_asymptotic} 
	For all
	$s>1/2$, it holds that 
	$$
	g_
	{m}^{\lin}( 
	H^{s}_{\mix}, L_{\infty}) 
	\,\asymp_{s, d} 
        m^{-s+1/2}\,   \left(  \log m \right)^{s(d-1)}.
	$$	
\end{theorem}

\medskip

This result is known, see~\cite{Tem93} or~\cite[Theorem~5.4.1]{DTU18}, 
and the upper bound is achieved by sparse grids. 
Here, we show that the same bound is achieved 
by a suitable least squares method, 
see Remark~\ref{rem:alg}.

\begin{proof}[Proof of Theorem~\ref{Th_Sobolev_type_asymptotic}]
It is known \cite{Bab,Mit} that for the singular numbers $\sigma_m = c_m(H^{s}_{\mix},  L_2)$ it 
 {holds that}
	$\sigma_m \lesssim_{s,d} m^{-s} (\log m)^{s(d-1)} \,$.
Hence, Theorem~\ref{Th_Sobolev_type_asymptotic} follows 
 immediately from 
Corollary~\ref{SC_general_weights}
   and the calculations:
$$
		\sum_{k
	>m}	\sigma_k^2
		\,\lesssim_{s,d}\,
		\sum_{k>m }  k^{-2s} (\log k)^{2s(d-1)}
		\,\lesssim_{s,d}\, m^{-2s+1}   \left(  \log m \right)^{2s(d-1)}.
$$
\end{proof}

\subsubsection{Sobolev- and 
Nikol'skii-Besov spaces
on the $d$-torus}

Here, we consider the unit balls of the mixed order Sobolev spaces $\mathbf{W}^r_p$ and 
Nikol'skii-Besov spaces $\mathbf{H}^r_p$ on the $d$-dimensional torus $D=\mathbb{T}^d$. 

For $\boldsymbol{s}=(s_1, \dots, s_d)$ with $s_j \in \mathbb{N}_0$, $j=1, \dots, d$, we define the sets
$$
\rho(\boldsymbol{s}):= \left\{ (k_1,\dots,k_d)\in \mathbb{Z}^d \colon \ \lfloor2^{s_j-1} \rfloor\leq \abs{k_j}< 2^{s_j}, \ j=1, \dots, d  \right\},
$$
and the so-called
step $\Omega_\ell$ and smooth $\Omega_\ell'$ hyperbolic crosses 
\[
  \Omega_\ell \,:=\, \bigcup_{\abs{\boldsymbol{s}}_1 \leq \ell }  \rho(\boldsymbol{s})
  \,
  \subset \, \bigg\{ (k_1,\dots,k_d) \in \mathbb{Z}^d\colon \prod_{j=1}^{d}\max \{1,  \abs{k_j}\}< 2^\ell   \bigg\} =: \Omega_\ell'.
\]

The cardinalities
$\abs{\Omega_\ell}$ and {$\abs{\Omega_\ell'}$ are}  
of order $2^\ell \ell^{d-1}$.
For $f \in L_1$ and $\bx\in\mathbb T^d$ we set 
$$
\delta_{\boldsymbol{s}} (f)(\bx) \,:=\,
\sum_{\bk \in \rho(\boldsymbol{s})} f_\bk \, e^{i \bk \cdot \bx}.
$$

We now define, for $1<p<\infty$ and $r>0$, the classes
 $$
\mathbf{H}^r_p \,:=\, \Big\{ f\in L_1
\colon \; \sup_{\boldsymbol{s}\in\mathbb{N}_0^d} \, 2^{r\abs{\boldsymbol{s}}_1}\,
\Vert \delta_{\boldsymbol{s}} (f) \Vert_{p}  \le1 \Big\}
$$
and 
 $$
\mathbf{W}^r_p \,:=\, \Bigg\{ f\in L_1
\colon \; \norm{\Big(\sum_{\boldsymbol{s}\in\mathbb{N}_0^d} \, 2^{2r\abs{\boldsymbol{s}}_1} 
\abs{\delta_{\boldsymbol{s}} (f)}^2 \Big)^{1/2}}_p \le1 \Bigg\}
$$
{of functions with bounded mixed 
smoothness.}
 The definitions of these classes that include the limit cases $p=1$ and $p=\infty$ can be found in \cite[Chapter~3]{DTU18}. 
It is known that the embedding 
$F\hookrightarrow L_q$,  {$1\le q \le \infty$,} where $F\in\{\mathbf{H}^r_p,  \mathbf{W}^r_p \}$, 
is compact under the condition 
$r> (1/p-1/q)_+$. However, we need function evaluation to be a well-defined continuous functional and the sequence   
$(\alpha_k)_{k =1}^\infty$
defined by \eqref{eq:an} to be square summable,  {and}
therefore restrict ourselves
 {to}
considering 
$r>\max \{1/2, 1/p\}$.
In this case, we have a continuous embedding of the respective class $F$ into $C(\mathbb{T}^d)$ and we identify each $f\in F$ with its unique continuous representative.

Figure~\ref{fig1} gives an overview of our findings for these classes.
{In particular, we contribute to some open problems from \cite{DTU18}, see Open Problem~5.3 as well as Figures 5.2, 5.3.
Here we removed question marks for the regions with $1\leq q\leq 2 \leq p< \infty$, where our least squares approach gives sharp bounds
for both of the classes
$\mathbf{H}^r_p$ and $\mathbf{W}^r_p$
}
For instance, we obtain that there is an unweighted least-squares algorithm that gives the optimal rate of convergence for $L_q$-approxi\-mation on the classes $\mathbf{W}^r_2$, universally for all $1\le q \le \infty$.

 {Let us formulate a result for $\mathbf{H}^r_p$.}
Here, 
$Q_\ell$ and $Q_\ell'$ denote 
the sets of trigonometric polynomials with frequencies from 
the respective
hyperbolic crosses, 
i.e.,
we put 
\begin{equation}\label{eq:Ql}
    Q_\ell := \operatorname{span} \{ e^{i \bk \cdot \bx } \colon \ \bk \in \Omega_\ell   \}, \quad Q_\ell' := \operatorname{span} \{ e^{i \bk \cdot \bx } \colon \ \bk \in \Omega_\ell'   \}.
\end{equation}

\medskip

\begin{corollary}\label{estim_nik-besov-classes}
Let $1<p<\infty$ and $r>\max \{1/2, 1/p\}$. For all $\ell\in\mathbb{N}$, there are linear sampling algorithms $A_m \colon \mathbf{H}^r_p \to Q_\ell$ 
that use at most 
$m \asymp \dim Q_\ell \asymp 2^\ell \ell^{d-1}$ points
and satisfy 
 {for all
$1\leq q 
<
\infty$} that
 $$
	\sup_{f\in{\mathbf{H}^r_p}}\, \Vert f - A_m f  \Vert_{q}
\,\lesssim\,  m^{-(r-t)}
(\log m)^{(d-1)\left(r-t+1/2\right)},
$$
where $t:=(1/p-1/2)_+ +(1/2-1/q)_+$.
{In the cases}
\begin{itemize}
    \item $1<p <\infty$ and $q=2$, or
    \item $1\leq q \leq2\leq p < \infty$,
\end{itemize}
the bound is sharp 
{($\lesssim$ can be replaced with $\asymp$) 
and asymptotically equivalent to the approximation numbers $a_m(\mathbf{H}^r_p, L_q)$.}

For 
uniform approximation, we have $t=\max\{1/2,1/p\}$ and
\[
 \sup_{f\in{\mathbf{H}^r_p}}\Vert f - A_m f  \Vert_\infty 
 \,\lesssim\, 	m^{-(r- t)} 
(\log m)^{(d-1)\left(r-t+ 1\right)}.
\]
\end{corollary}

\begin{proof}[Proof of Corollary~\ref{estim_nik-besov-classes}]
All of the conditions of Corollary~\ref{coro:general} are satisfied. Indeed, we consider the nested sequence of subspaces $Q_\ell$, $\ell \in \mathbb{N}$, 
see~\eqref{eq:Ql}. Then, the projection $P_{Q_\ell}$ of $f \in \mathbf{H}^r_p$ takes the form
$$
P_{Q_\ell} f = \sum_{\bk \in \Omega_\ell} f_\bk  e^{i \bk \cdot \bx}.
$$   
The quantity 
$$
\sup_{f \in \mathbf{H}^r_p} \Vert f - P_{Q_\ell} f  \Vert_2 
$$
is nothing else than the best approximation $\mathcal{E}_{Q_\ell} (\mathbf{H}^r_p)_2$ of functions from $\mathbf{H}^r_p$ by the step hyperbolic cross Fourier sums $S_{Q_\ell}$. We know that
{it is equal to} 
the best approximation  by the step hyperbolic cross polynomials. 
Therefore, we can use the estimate
$$
{\mathcal{E}_{Q_\ell} (\mathbf{H}^r_p)_2}
\asymp 2^{-\ell (r-(\frac{1}{p}-\frac{1}{2})_+)}  {\ell}^{\frac{d-1}{2}}\,, \quad 1<p<\infty,
$$
 see \cite[Thm. 4.2.6]{DTU18} and the references therein.
 
Since $\dim{Q_\ell}\asymp 2^\ell \ell^{d-1}$, we put  $\alpha_{2^\ell \ell^{d-1}} := 2^{-\ell (r-(1/p-1/2)_+)} \ell^{(d-1)/2}$, and hence
$$
\alpha_m \,\asymp\,
m^{- (r-(\frac{1}{p}-\frac{1}{2})_+)}
(\log m)^{(d-1)\left(r-(\frac{1}{p}-\frac{1}{2})_+ + \frac{1}{2}\right)}.
$$

Using the Nikol'skii inequality (see \cite[Thm.~2.4.9]{DTU18} with $Q_{\ell}\subset 
{Q'_\ell}$), we get that 
$$
{\Lambda^{(q)}_{Q_{\ell}}}:= 
\sup_{f\in Q_\ell, f \neq 0}\frac{\|f\|_q}{\|f\|_2} \lesssim 2^{\ell (\frac12-\frac1q)}, \quad 2\leq q < \infty.
$$
Therefore, in view of $\dim Q_\ell\asymp2^{\ell} \ell^{d-1}$, it 
 {holds that}
\begin{equation}\label{Chr_f_Lq_trigonom}
\Lambda^{(q)}_n \lesssim n^{(\frac12-\frac1q)_+} (\log n)^{(d-1)(\frac1q-\frac12)_-}, \quad 1\leq q < \infty,
\end{equation}
where $a_- = \min\{a, 0\}$.

Then, Corollary~\ref{coro:general} with
$$\alpha = r-(1/p-1/2)_+ 
> \max\{1/2, (1/2-1/q)_+ \}= 1/2, $$
which holds for 
$r>\max \{1/2, 1/p\}$, 
as well as 
$$\gamma =(d-1)(r-(1/p-1/2)_+ + 1/2), $$ 
$\beta = (1/2-1/q)_+$ and $\delta =(d-1)(1/q-1/2)_-$
yields that 
there exist linear sampling algorithms $A_{m} \colon \mathbf{H}^r_p \to Q_\ell$,
with	$m \asymp \dim Q_\ell \asymp 2^\ell \ell^{d-1}$ 
 {points,}
satisfying
$$
\sup_{f \in \mathbf{H}^r_p} \Vert f - A_{m} f  \Vert_q	 \lesssim 
m^{-(r-(\frac{1}{p}-\frac{1}{2})_+) + (\frac12-\frac{1}{q})_+}
(\log m)^{(d-1)\left(r-(\frac{1}{p}-\frac{1}{2})_+ + \frac{1}{2} + (\frac{1}{q}-\frac{1}{2})_-\right)},
$$
where 
$1\leq q < \infty$.
 {Note that}
the condition $r>\max \{1/2, 1/p\}$ ensures the compact embedding of $\mathbf{H}^r_p$ in $L_q$ for all $p$, $q$ {(it is stronger than the condition $r>(1/p-1/q)_+$)}, so the problem is well defined.

{Equation~  {(\ref{Chr_f_Lq_trigonom})} does not extend to the case $q=\infty$.
Here, we only have}
\begin{equation}\label{Chr_f_Linfty_trigono}
\Lambda^{(\infty)}_n \,\lesssim\, n^{\frac12},
\end{equation}
(see \cite[Thm.~2.4.8]{DTU18}),
and hence {$\delta=0$, so that} 
$$
\sup_{f \in \mathbf{H}^r_p} \Vert f - A_{m} f  \Vert_\infty	 \lesssim 
m^{-(r-(\frac{1}{p}-\frac{1}{2})_+) + \frac12}
(\log m)^{(d-1)\left(r-(\frac{1}{p}-\frac{1}{2})_+ + \frac{1}{2} \right)}.
$$

 {In order to discuss sharpness of the obtained upper bounds we take into account the relation 
$$
a_m(\mathbf{H}^r_p, L_q) \leq g_m^{\lin} (\mathbf{H}^r_p, L_q) \leq 
\sup_{f \in \mathbf{H}^r_p} \Vert f - A_{m} f  \Vert_q\,.	
$$
In the cases $q=2$, $1<p <\infty$ and $1\leq q \leq2\leq p < \infty$ one has (see, e.g., \cite[Thm. 4.5.2]{DTU18}) the matching lower bound for approximation numbers $a_m(\mathbf{H}^r_p, L_q)$ that yields sharp orders for the least squares recovery.  
}
\end{proof}

\subsubsection*{Left Figure~\ref{fig1}}
{
\medskip
{\bf Green area} ($1\leq q \leq 2\leq p
<\infty$).  We obtained the right orders. In this parameter region,
the exact orders of recovery
for Smolyak-type algorithms \cite[Thm.~5.1]{Dung_Ullrich2015} do not match the bounds for the linear widths (even in the case $p=q=2$). 

{\bf Magenta line} ($p=2<q<\infty$).  The orders of sampling numbers and linear widths coincide and are equal to 
$$m^{-(r-(1/2-1/q))} (\log m)^{(d-1)(r-1/p+2/q)}$$ (see \cite[Cor.\,~7.6]{Byrenheid_Ullrich2017}). 
In this parameter region, our approach leads to worse bounds in a logarithmic scale, but we make a conjecture that the least squares algorithm should also give the sharp order (as it holds for the $\mathbf{W}^r_p$ classes, see Corollary~\ref{estim_Sobolev_classes} below). 

{\bf Blue area} ($2\leq p<q < \infty$ and $1<p<q\leq 2$). The obtained upper bounds for the least squares recovery are worse in the main rate than those (optimal in order) for the Smolyak algorithm.
 {Let us also recall that} the first sharp order for the function recovery using the Smolyak algorithm on the classes $\mathbf{H}^r_p$ in $L_q$ for $1<p<q\leq 2$ was obtained
in \cite{DD91}. 

{\bf White area.} For the triangular regions, where either $1<q\leq p <2$ or $2<q\leq p < \infty$,  our bounds for the least squares recovery are worse in the main rate than the best known (not optimal in order) for Smolyak algorithms. 

For the lower right region ($1<p<2<q<\infty$),  the upper bounds for the least squares recovery are worse only in the logarithmic rate than the respective best known bounds for the Smolyak algorithm.
Using the ``fooling argument'' \cite[Cor.\,8.6]{Byrenheid_Ullrich2017}, one can show that in this region the
main rate $m^{-(r-(1/p-1/q))}$ for the sampling numbers $g_n^{\lin}(\mathbf{H}^r_p, L_q)$ is optimal. 
However, in this range the linear widths {decay faster} 

than the linear sampling numbers.

{\bf Case $\boldsymbol{q=\infty}$.}
    Let us compare the estimate of Corollary~\ref{estim_nik-besov-classes} with the
bound $m^{-(r-1/p)}$ $(\log m)^{(d-1)(r-1/p+1)}$ for  the 
{$L_\infty$}-recovery 
of functions from $\mathbf{H}^r_p$, $1<p<\infty$, $r>1/p$, on the Smolyak grids  (see \cite{DD91}, \cite[Thm.~5.3.3 (ii)]{DTU18}  and the references therein). 
We 
 {see that}
for $1<p\leq 2$ the 
corresponding estimates coincide. 
 {Note that}
it is not known whether these bounds are sharp. However, for $2<p<\infty$, Smolyak's algorithm gives bounds that are better even in the main rate.

\medskip

For the Sobolev classes $\mathbf{W}^r_p$
we obtain the following.

\begin{corollary}\label{estim_Sobolev_classes}
Let $1< p <\infty$ and $r>\max \{1/2, 1/p\}$. For all $\ell\in\mathbb{N}$, there are linear sampling algorithms $A_m \colon \mathbf{W}^r_p \to Q_\ell$ 
that use at most $m \asymp \dim Q_\ell \asymp 2^\ell \ell^{d-1}$ points and
	satisfy
for all
$1\leq q < \infty$ that
$$
\sup_{f\in{\mathbf{W}^r_p}}\Vert f - A_m f  \Vert_q
	\,\lesssim\, m^{ -(  r-t)} 
(\log m)^{(d-1)(r-t) }\,,
	$$
where $t=(1/p-1/2)_+ +(1/2-1/q)_+$.

For uniform approximation, we have $t=\max\{1/2,1/p\}$ and
\[
 \sup_{f\in{\mathbf{W}^r_p}}\Vert f - A_m f  \Vert_\infty 
 \,\lesssim\, m^{ -( r-t)} (\log m)^{(d-1)(r-t+1/2)}.
\]
{In the cases}
\begin{itemize}
    \item $1<p <\infty$ and $q=2$, or
    \item $1\leq q \leq2\leq p < \infty$, or
    \item {$p=2\leq q \le \infty$}.
\end{itemize}
those bounds are sharp {($\lesssim$ can be replaced with $\asymp$) } and 
{asymptotically equivalent to the approximation numbers $a_m(\mathbf{W}^r_p, L_q)$.}
\end{corollary}

\begin{proof}[Proof of Corollary~\ref{estim_Sobolev_classes}]

Arguing similarly 
as in the proof of
Corollary \ref{estim_nik-besov-classes} for the Nikol'skii-Besov classes, we will use the estimate for the corresponding best approximation by the step hyperbolic cross polynomials 
$$
{\mathcal{E}_{Q_
\ell}
}(\mathbf{W}^r_p)_2 \asymp 2^{-
{\ell}
(r-(\frac{1}{p}-\frac{1}{2})_+)} \,, \quad r>
{(1/p-1/2)_+,}
\ 1< p <\infty,
$$
see \cite[Thm. 4.2.4]{DTU18} and the references therein.
 
Hence, we put $\alpha_{2^{\ell} {\ell}^{d-1}} := 2^{-{\ell} (r-(1/p-1/2)_+)}$,
and therefore
$$\alpha_m \,\asymp\, m^{ -( r-(\frac{1}{p}-\frac{1}{2})_+)} (\log m)^{(d-1)(r-(\frac{1}{p}-\frac{1}{2})_+)}.$$

{Corollary~\ref{coro:general}}
with 
$\alpha  = r-(1/p-1/2)_+$, 
$\gamma =(d-1)(r-(1/p-1/2)_+)$, 
$\beta = (1/2-1/q)_+$ and $\delta =(d-1)(1/q-1/2)_-$\,, 
$1\leq q <\infty$ (see (\ref{Chr_f_Lq_trigonom})), and, respectively, 
$\beta = 1/2$, $\delta =0$ in the case $q=\infty$ (see (\ref{Chr_f_Linfty_trigono}))
 yields the 
  {upper bounds in}
  Corollary~\ref{estim_Sobolev_classes}.

 {Sharpness of the obtained estimates for the indicated  relations between the parameters $p$ and $q$ follows from the matching lower bounds for the approximation numbers 
$a_m(\mathbf{W}^r_p, L_q)$ (see, e.g., \cite[Thm. 4.5.1]{DTU18}).
}
\end{proof}

\subsubsection*{Right Figure~\ref{fig1}}
{\medskip
{\bf Green area} ($1\leq q \leq 2\leq p
<\infty$).  We obtained the asymptotically sharp orders. In this parameter region,
the so far known bounds for Smolyak-type algorithms do not match the bounds for the linear widths (even in the case $p=q=2$, see also~\cite{DKU}). 

{\bf Magenta lines} ($p=2<q<\infty$ and $1 < p < 2 = q$).  The obtained orders of the least squares recovery coincide with
the orders of linear widths. The same was observed also for the Smolyak algorithm. In \cite{MPU25} the discussion is extended to non-linear algorithms, where in this case non-linear sampling outperforms linear sampling in the horizontal magenta line, i.e., $1<p<2=q$.

{\bf Blue area} ($2\leq p<q < \infty$ and $1<p<q\leq 2$). The upper bounds for the least squares recovery are worse in the main rate than those (optimal in order) for the Smolyak algorithm.

{\bf White area.} For the triangular regions, where either $1<q\leq p <2$ or $2<q\leq p < \infty$,  our bounds for the least squares recovery are worse in the main rate than the sharp orders of linear widths. 

For the lower right region ($1<p<2<q<\infty$),  our approach gives the same upper bound  as the linear Smolyak-type interpolation \cite[Cor.\,7.1]{Byrenheid_Ullrich2017} 
(this is not the case for the $\mathbf{H}^r_p$ classes).  
Similarly as for the $\mathbf{H}^r_p$ classes, using the ``fooling argument'' \cite[Cor.\,8.6]{Byrenheid_Ullrich2017}, one can show that if $1<p<2< q<\infty$, the
main rate $m^{-(r-(1/p-1/q))}$ for the sampling numbers $g_n^{\lin}(\mathbf{W}^r_p, L_q)$ is optimal. 
However, in this range the linear widths 
{decay faster} 
than the linear sampling numbers.

{\bf Case $\boldsymbol{q=\infty}$.}
 For the $L_\infty$-recovery of
   $f\in \mathbf{W}^r_p$, $1<p<\infty$, $r>1/p$ on the Smolyak grids, the upper bound
    $m^{-(r-1/p)}$ $(\log m)^{(d-1)(r-2/p+1)}$ is known \cite[Cor.\ 7.8]{Byrenheid_Ullrich2017}.
Hence, the deterministic Smolyak algorithm outperforms least-squares 
in both of the cases $p<2$ and $p>2$. For $p=2$ the respective bounds coincide and are asymptotically sharp \cite{Tem93}. 
   
}

\medskip

\subsubsection{Wiener-type and Korobov spaces}

For $r>0$, let
$$
\mathcal{A}^r_{\mix} := \Big\{ f\in L_1(\mathbb{T}^d) 
\colon \; \sum_{\bk \in \mathbb{Z}^d} \abs{f_{\bk}} \cdot \prod_{j=1}^{d}\max \{1,  \abs{k_j}\}^{r}  \le1 \Big\}
$$
be a mixed Wiener class,
where $f_{\bk}$ are the Fourier coefficients of $f$ w.r.t.\ the trigonometric system.
We obtain 
\begin{corollary}\label{estim_Wiener_linear}
	For $r>1/2$, there are linear sampling algorithms $A_m \colon \mathcal{A}^r_{\mix} \to Q_\ell'$ with 
	$Q_\ell'$ from~\eqref{eq:Ql} and $m \asymp \dim Q_\ell' \asymp 2^\ell \ell^{d-1}$ 
     {points,}
    satisfying 
	$$
	\sup_{ f \in {\mathcal{A}^r_{\mix}}} \Vert f - A_m f  \Vert_q
	\lesssim 
	m^{-(r - (\frac12-\frac{1}{q})_+)} {(\log m)^{(d-1)(r- (\frac12-\frac{1}{q})_+ 
	)}}
	$$
	for every
	{$1\leq q < \infty$.
In particular, 
	it 
     {holds that}
$$
	\sup_{f \in {\mathcal{A}^r_{\mix}}} \Vert f - A_m f  \Vert_2
 	\asymp 
	m^{-r} (\log m)^{(d-1)r}.
	$$
For the $L_\infty$-bounds,	we get
$$
	\sup_{f \in {\mathcal{A}^r_{\mix}}} \Vert f - A_m f  \Vert_\infty 
 	\lesssim 
	m^{-(r - \frac{1}{2})} (\log m)^{(d-1)r}.
$$}
\end{corollary}

\begin{proof}[Proof of Corollary~\ref{estim_Wiener_linear}]
For $f \in \mathcal{A}^r_{\mix}$, we have that 
\begin{align*}
&\Vert f  - P_{Q_\ell'} f  \Vert_2  
=	\Vert  \sum_{\bk\in \mathbb{Z}^d \setminus \Omega_m'} f_{\bk} e^{i \bk \cdot \bx }  \Vert_2
= 	\left( \sum_{\bk\in \mathbb{Z}^d \setminus \Omega_m'} \abs{f_{\bk}}^2 \right)^{1/2}
	\\
&	=  \left( \sum_{\bk\in \mathbb{Z}^d \setminus \Omega_m'} \abs{f_{\bk}}^2  \prod_{j=1}^{d} (1+ \abs{k_j})^{2r} \prod_{j=1}^{d} (1+ \abs{k_j})^{-2r} \right)^{1/2}
	\\
&	\leq \sup_{\bk \in \mathbb{Z}^d \setminus \Omega_m'} \prod_{j=1}^{d} (1+ \abs{k_j})^{-r}
	\left( \sum_{\bk\in \mathbb{Z}^d \setminus \Omega_m'} \abs{f_{\bk}}^2  \prod_{j=1}^{d} (1+ \abs{k_j})^{2r} \right)^{1/2}
	\\
&	\leq  2^{-\ell r} \left( \sum_{\bk\in \mathbb{Z}^d \setminus \Omega_m'} \abs{f_{\bk}}  \prod_{j=1}^{d} (1+ \abs{k_j})^{r} \right)^{1/2} 
 \leq 2^{-\ell r} .
\end{align*}

Hence, since  $\dim{Q_\ell'} \asymp 2^\ell \ell^{d-1}$, we put $m \asymp 2^\ell \ell^{d-1} $ and get 
$$
\sup_{f \in \mathcal{A}^r_{\mix}} \Vert f - P_{V_m} f  \Vert_2\lesssim  m^{-r} (\log m)^{(d-1)r}.
$$
Taking into account the bound  (\ref{Chr_f_Lq_trigonom}) for the Christoffel function, 
from 
{Corollary~\ref{coro:general}}
with
$\alpha = r >1/2$, $\gamma =(d-1)r$, 
$\beta = (1/2-1/q)_+$ and $\delta =(d-1)(1/q-1/2)_-$
we get that
there are linear sampling algorithms $A_m \colon \mathcal{A}^r_{\mix}\to Q_\ell'$,
 $r>1/2$, with $m \asymp \dim Q_\ell' \asymp 2^\ell \ell^{d-1}$
  {points,}
satisfying
$$
\sup_{f \in \mathcal{A}^r_{\mix}} \Vert f - A_m f  \Vert_q \lesssim 
m^{-(r - (\frac12-\frac{1}{q})_+)} 
(\log m)^{(d-1)(r+ (\frac{1}{q}-\frac{1}{2})_-)},
$$
where 
$1\leq q <\infty$.

If $q=\infty$, (\ref{Chr_f_Linfty_trigono}) yields the bound
$$
\sup_{f \in \mathcal{A}^r_{\mix}} \Vert f - A_m f  \Vert_\infty \lesssim 
m^{-(r - \frac12)} 
(\log m)^{(d-1)r}.
$$%
\end{proof}

\medskip

\begin{remark}
From \cite[Thm. 5.1 (Remark\ 6.4)]{Kol_Lomako_Tikhonov23} 
(putting $\alpha =r$, 
 $p=1$, $\sigma_{1,q'}=0$
and
$m\asymp 2^n n^{d-1}$), we get the upper bound $m^{-r} (\log m)^{(d-1)r}$ for the approximation of functions from 
 the weighted
Wiener spaces by sparse grid methods 
 in $L_q$, $2\leq q \leq \infty$.
\\
This bound in the case $q=2$ coincides with those obtained in our Corollary~\ref{estim_Wiener_linear}. 
 {The estimate}
is asymptotically sharp as the recent results in \cite{NNS} on the Kolmogorov numbers for this embedding show.
However, Smolyak-type algorithms give better main rate estimates  {for
$2<q\leq \infty$.}
\end{remark}

Let us now consider the Korobov classes
$$
\mathbf{E}^r_d  := \left\{ 
f\in L_1(\mathbb{T}^d) \colon  \sup_{\bk \in \mathbb{Z}^d } \left( \abs{f_{\bk}} \cdot \prod_{j=1}^{d}\max \{1,  \abs{k_j}\}^{r} \right) \leq 1  \right\}.
$$
We obtain the following.

\begin{corollary}\label{estim_Korobov_linear_L_2}
	For $r>1$, there are linear sampling algorithms $A_m \colon \mathbf{E}^r_d \to Q_\ell'$ with $Q_\ell'$ from~\eqref{eq:Ql} and $m \asymp \dim Q_\ell' \asymp 2^\ell \ell^{d-1}$
     {points}
    satisfying
	$$
\sup_{f\in \mathbf{E}^r_d} \Vert f - A_m f  \Vert_q \lesssim 
m^{-(r -  \frac{1}{2} - (\frac{1}{2}-\frac{1}{q})_+)}
(\log m)^{(d-1)(r - (\frac{1}{2}-\frac{1}{q})_+)},
$$
for every 
{$1\leq q < \infty$.
In particular, 	it 
 {holds that}
$$
\sup_{f\in \mathbf{E}^r_d}\,  \Vert f - A_m f  \Vert_2 
\,  \asymp\, m^{-(r - \frac{1}{2})} (\log m)^{(d-1)r} \, .
	$$
For the $L_\infty$-bounds,	we get
$$
\sup_{f\in \mathbf{E}^r_d}\,  \Vert f - A_m f  \Vert_\infty 
	\,  \lesssim\, m^{-(r -1)} (\log m)^{(d-1)r}.
$$
}
\end{corollary}

Note that we obtain the same bound (but with a different method) for uniform approximation of functions from the Korobov classes as obtained by Smolyak in his famous 1963 paper where he introduced sparse grid methods, also known as Smolyak's algorithm, see \cite[p.~1044]{Sm}.

\begin{proof}[Proof of Corollary~\ref{estim_Korobov_linear_L_2}]
Using similar arguments as above, we get
\begin{align*}
	\sup_{f \in \mathbf{E}^r_d} & \Vert f - P_{Q_\ell'} f  \Vert_2
	\\ 
&	=
		\sup_{f \in \mathbf{E}^r_d} \left(
		\sum_{\bk\in \mathbb{Z}^d \setminus \Omega_\ell'}
		\abs{f_{\bk}}^2 \cdot \prod_{j=1}^{d}\max \{1,  \abs{k_j}\}^{2r} 
		\prod_{j=1}^{d}\max \{1,  \abs{k_j}\}^{-2r} 
		\right)^{1/2}
		\\
& 
\leq 
	\left( \sum_{\bk\in \mathbb{Z}^d \setminus \Omega_\ell'}
	\prod_{j=1}^{d}\max \{1,  \abs{k_j}\}^{-2r}
	\right)^{1/2}
\\
& =\left( \sum_{\abs{\boldsymbol{s}}_1 > \ell}  \sum_{\bk\in \rho(\boldsymbol{s})} 
	\prod_{j=1}^{d}\max \{1,  \abs{k_j}\}^{-2r}
	\right)^{1/2}
	\leq \left( \sum_{\abs{\boldsymbol{s}}_1 > \ell}  2^{-2r \abs{\boldsymbol{s}}_1} \sum_{\bk\in \rho(\boldsymbol{s})} 1
	\right)^{1/2} 
			\\
& \lesssim \left( \sum_{\abs{\boldsymbol{s}}_1 \geq \ell}  2^{-2r \abs{\boldsymbol{s}}_1}  \cdot 2^{\abs{\boldsymbol{s}}_1}
	\right)^{1/2} 
	\lesssim  \left( 2^{- (2r-1)\ell}  \ell^{d-1}\right)^{1/2} 
= 2^{-(r-\frac{1}{2})\ell} \ell^{\frac{d-1}{2}}.
\end{align*}

Hence, from 
{Corollary~\ref{coro:general}}
with $\alpha_m = m^{-(r-1/2)} (\log m)^{(d-1)r}$, $\alpha = r >1/2$, $\gamma =(d-1)r$,
$\beta = (1/2-1/q)_+$ and $\delta =(d-1)(1/q-1/2)_-$, we get that
there are linear sampling algorithms $A_m \colon \mathbf{E}^r_d  \to Q_\ell'$,
 $r>1/2$, with $m \asymp \dim Q_\ell' \asymp 2^\ell \ell^{d-1}$
  {points,}
satisfying
$$
\sup_{f \in \mathbf{E}^r_d} \Vert f - A_m f  \Vert_q \lesssim 
m^{-(r -  \frac{1}{2} - (\frac12-\frac{1}{q})_+)} (\log m)^{(d-1)(r +(1/q-1/2)_- )},
$$
where $1\leq q 
<\infty$.

The corresponding lower bound for the $L_2$-approximation follows from the embedding $\mathbf{H}^r_1 \subset \mathbf{E}^r_d $ and the estimate for the Kolmogorov widths of the classes $\mathbf{H}^r_1$, $r>1$, in $L_2$, see \cite[Thm. 4.3.12]{DTU18}.

If $q=\infty$, we put $\beta = 1/2$, $\delta =0$ and hence get
$$
\sup_{f \in \mathbf{E}^r_d} \Vert f - A_m f  \Vert_\infty \lesssim 
m^{-(r-1)} (\log m)^{(d-1)r }.
$$
\end{proof}

\begin{remark}
In the case of $L_q$-approximation with 
$2 {<}
q < \infty$, 
one can derive (better) 
bounds for $f\in \mathbf{E}^r_d$  from the embedding $\mathbf{E}^r_d \subset \mathbf{H}^{r-1/2}_2$. 
The corresponding estimate $m^{-(r-1+1/q)} (\log m)^{(d-1)(r-1+2/q)}$ for Smolyak-type algorithms was obtained in \cite[Thm. 6.6]{Byrenheid_Ullrich2017} 
(putting there $r_1 = r-1/2$, $p=2$, $\theta = \infty$, $\mu = d$, $m\asymp 2^n n^{d-1}$). 
For $q=2$, the bounds coincide and are asymptotically sharp. 
 {Note that}
\cite[Thm. 5.1 (Remark\ 6.4)]{Kol_Lomako_Tikhonov23} (putting $\alpha =r$, $p=\infty$,  
$\sigma_{\infty, q'}=1-1/q$ and $m\asymp 2^n n^{d-1}$) also yields the upper bound $m^{-(r-1+1/q)} (\log m)^{(d-1)(r+ 1/q)}$ for 
Smolyak-type recovery in $L_q$, $2<q< \infty$, which is worse in the logarithmic rate even than those obtained in our Corollary~\ref{estim_Korobov_linear_L_2}.
\end{remark}


\subsection{Sobolev spaces on manifolds}
\label{subsec:mani} \,

Let $D$ be a compact homogeneous Riemannian manifold
equipped with the normalized volume measure $\mu$.
Typical examples are the torus and the sphere.
We let $0= \lambda_0 < \lambda_1 < \dots$ be the eigenvalues of the Laplace-Beltrami operator and $E_0, E_1, \dots$ be the corresponding eigenspaces. The eigenspaces are known to be finite-dimensional and orthogonal in $L_2(\mu)$. We take orthonormal bases $(b_{\ell,k})_{k\le \dim(E_\ell)}$ of the spaces $E_\ell$. By the addition theorem, see \cite[Thm.~3.2]{Gin},
we have
\[
 \sum_{k=1}^{\dim(E_\ell)} |b_{\ell,k}(x)|^2 \,=\, \dim(E_\ell),
 \quad x\in D.
\]
We define the generalized Sobolev spaces $H^w=H^w(D)$ via the reproducing kernel
\[
 K(x,y) \,=\, \sum_{\ell=0}^\infty w_\ell^2 \sum_{k=1}^{\dim(E_\ell)} b_{\ell,k}(x) \overline{b_{\ell,k}(y)}.
\]
where we assume that the sequence $(w_\ell)$ is non-increasing and satisfies
\begin{equation}\label{eq:beta}
 \sum_{\ell=0}^\infty w_\ell^2 \dim(E_\ell) \,<\, \infty.
\end{equation}
This condition 
is needed for the kernel to be well-defined
and ensures the boundedness of the kernel at the same time.
The classical Sobolev spaces $H^s=H^s(D)$ of smoothness $s$ are obtained
if we put $w_{\ell}^2 = \lambda_\ell^{-s}$
and then the relation \eqref{eq:beta} is satisfied if and only if $s>\dim(D)/2$, see again \cite{Gin} and the references therein.

To apply our results, we observe that
\[
 \Lambda_m^2 = m \quad\text{and}\quad K_m(x,x)\,=\, \sum_{\ell>L}^\infty w_\ell^2\, \dim(E_\ell)
\]
for all $m$ of the form 
\begin{equation}\label{eq:special-m}
 m = \sum_{\ell\le L} \dim(E_\ell), \quad  L\in\mathbb{N}_0.
\end{equation}
We obtain universal constants $C,c>0$ such that,
for all $m$ as above 
and $2\le p \le \infty$,
\begin{equation}\label{eq:UB-mani-p}
 g_{cm}^\lin(H^w,L_p) \,\le\,  C m^{-1/p} \max\bigg\{ \sqrt{m}\,w_{L+1}, \, \sqrt{\sum_{\ell>L}^\infty w_\ell^2\, \dim(E_\ell)} \bigg\}.
\end{equation}
{where, of course, $1/p:=0$ for $p=\infty$.
Indeed, the case $p=2$ follows from the $L_2$-error bound \eqref{eq:DKU} from \cite[Thm.~23]{DKU}, and then the case $p=\infty$ follows from Lemma~\ref{lem:liftH}.
Since both bounds are for the same algorithm, we can interpolate to obtain the bound for the remaining cases $p\in(2,\infty)$.}

On the other hand, the approximation numbers (as defined in \eqref{appnu}) satisfy $a_m(H,L_2)=w_{L+1}$ and 
\[
a_m(H^w,L_\infty) 
 \,=\, \sqrt{\sum_{\ell>L}^\infty w_\ell^2\, \dim(E_\ell)}
 {\,=\, \sqrt{\sum_{k>m} a_k(H,L_2)^2}}.
\]
{The lower bound in the previous equality follows from \eqref{cobos}
and the upper bound by taking the $L_2$-orthogonal projection onto $\bigcup_{\ell \le L} E_\ell$ as an approximation operator.
Therefore,
equation \eqref{eq:UB-mani-p} 
for $p=\infty$} 
alternatively reads
\begin{equation}\label{eq:UB-mani-inf}
 g_{cm}^\lin(H^w,L_\infty) \,\le\, C \max\big\{ \sqrt{m} \cdot a_m(H^w,L_2) ,\,  a_m(H^w,L_\infty) \big\}.
\end{equation}
{If we further define $m'$ to be the largest number of the form~\eqref{eq:special-m}
that satisfies $m'\le m/2$,
we may use
\[
 m \cdot a_m(H^w,L_2)^2 \,\le\, 2 \sum_{k=m'+1}^m a_k(H^w,L_2)^2
 \,\le\, 2 a_{m'}(H^w,L_\infty)^2
\]
to obtain from \eqref{eq:UB-mani-inf} that
\[
 g_{cm}^\lin(H^w,L_\infty) \,\le\, C' a_{m'}(H^w,L_\infty)
 \,=\, C' c_{m'}(H^w,L_\infty)
\]
with $C'=\sqrt{2} C$.}
In this sense,
standard information (function values) is as powerful as linear information for the problem of uniform approximation for all generalized Sobolev spaces on compact homogeneous Riemannian manifolds.

We also note that 
the dimension of $E_\ell$ increases at most polynomially (see Weyl's formula, {\cite[eq.\,(1)]{Gin}})
such that the relation \eqref{eq:UB-mani-p} in fact holds for all $m\in\mathbb{N}$ 
and a modified constant $c$, 
now depending on the dimension of the manifold $D$.
It can also be deduced from Weyl's formula that
\[
a_m(H^s,L_\infty) \,\asymp_{D,s}\, m^{-s/\dim(D)+1/2},
\]
and
\[
a_m(H^s,L_2) \,\asymp_{D,s}\, m^{-s/\dim(D)},
\]
such that we obtain the following.

\medskip

\begin{corollary}
    There are universal constants $c,C>0$ such that for any $d\in\mathbb{N}$, any compact homogeneous Riemannian manifold $D$ of dimension $d$, all Sobolev spaces $H^w=H^w(D)$ as above, and all $2\le p \le \infty$,
    we have 
    \[
     g_{cm}^\lin(H^w,L_p) \,\le\, C m^{-1/p} \max\big\{ \sqrt{m} \cdot a_m(H^w,L_2) ,\,  a_m(H^w,L_\infty) \big\},
    \]
    whenever $m$ equals the number of eigenvalues of the Laplace-Beltrami operator smaller than some $T>0$ (counted with multiplicity).
    In particular, for all $1\le p \le \infty$ and $s>d/2$,
    \[
    g_m^\lin(H^s,L_p) 
    \;\lesssim_{d,s}\; 
    m^{-s/d+(1/2-1/p)_+}.
    \]
\end{corollary}

\begin{remark}
The asymptotic relation
\[
 g_m^\lin(H^s,L_p) \;\asymp_{d,s}\; m^{-s/d+(1/2-1/p)_+}
\]
is known.
This is a classical result for Sobolev spaces on the cube (see, e.g., \cite[Chapter~6]{Heinrich})
and can be transferred to smooth manifolds, e.g., as in \cite{KSmani}.
We refer to \cite[Section~3.3]{HM} and \cite{FHJU22,Gr19,LS,Mhaskar,Mhaskar2} for more direct approaches
to function recovery on manifolds.
Our contributions are the alternative proof,
the dimension-independent constants
and a simple almost-optimal algorithm, see below.
\end{remark}

\smallskip

If we insert the algorithm from \cite[Thm.~1]{KU1} instead of \cite[Thm.~23]{DKU} into our Lemma~\ref{lem:liftH},
we obtain the following slightly worse
but more practical error bound.
Namely, for all $m=\dim(W_L)$, where $W_L$ denotes the span of the eigenspaces $E_0,\dots,E_L$, the least-squares algorithm
\begin{equation}\label{eq:algran}
 A_m(f) \,=\, \underset{g \in W_L}{\rm argmin}\, \sum_{i=1}^n |f(x_i)-g(x_i)|^2,
\end{equation}
using $n=c_t m\log(m)$ i.i.d.\ samples $x_1,\dots,x_n\in D$ distributed according to the volume measure $\mu$,
satisfies {with probability $1-n^{-t}$} 
that
\[
 \Vert f - A_m f \Vert_p \,\le\, 
 C m^{-1/p} \max\big\{ \sqrt{m} \cdot a_m(H^w,L_2) ,\,  a_m(H^w,L_\infty) \big\}
\]
for all $f\in H^w$ and $2\le p \le \infty$.
Here, {$t>0$ is arbitrary,
$C>0$ is a universal constant, and $c_t>0$ only depends on $t$.
See \cite[Thm.\,2]{U2020} for the probability estimate.}

In particular, we see that i.i.d.\ random samples are optimal up to a logarithmic oversampling factor for uniform approximation on such manifolds. This observation was already made in \cite{KSmani}. We improve upon the corresponding observation in \cite{KSmani} in the sense of the dimension-independent constants, the more general sequence $w$, and a possibly simpler algorithm~$A_m$.

Recall that the logarithmic oversampling factor can be removed by using only a subset of the random points in \eqref{eq:algran}, but the corresponding subsampling approach from \cite{DKU} is non-constructive.
We note that there is also a constructive way to (not remove but) reduce the logarithmic oversampling factor using the subsampling approach from \cite{BSU}.

Finally, let us mention that our results can also be applied for non-homogeneous Riemannian manifolds as the asymptotic bound $\Lambda_m \lesssim \sqrt{m}$ remains true in this setting, see \cite[Lemma~5.2]{Mhaskar2}.


\subsection{A non-uniformly bounded basis system} \label{ex:unbounded}

{Let us} define a second order operator $A$ on $[-1,1]$  by
  {$Af = - (gf')'$ with $g(x):=1-x^2$}. 
Together with the 
{normalized Lebesgue measure}
$\rm d x$ on $[-1,1]$, it
characterizes for $s>1$ weighted Sobolev spaces
\begin{equation*}
H(K_s) := \{ f\in L_2\colon \  A^{s/2}f \in L_2 \}\,.
\end{equation*}
These spaces are RKHS with the kernel
$$
K_s(x,y) = \sum_{k\in \mathbb{N}} \frac{{p}_k(x)  {p}_k(y)}{ 1+ (k(k+1))^s  } \,,
$$
where ${p}_k\colon [-1,1] \rightarrow \mathbb{R}$, $k \in 
{\mathbb{N}_0}$, are
$L_2$-normalized Legendre polynomials.  The polynomials $p_k$ are eigenfunctions of $A$ with eigenvalues $\lambda_k=k(k+1)$.
 {Let also $\Pi_m := \operatorname{span}\{p_k\colon \ 0\leq k\leq m-1 \}$.}

We refer to 
{\cite[Section 5]{Chernov_Dung2016} for the details}
{and to \cite{Bernardi_Maday1992}} for
the first $L_2$-approximation result 
using
Gauss points.  {For the $L_{\infty}$-approximation, we get the following estimate.}

\begin{theorem}
For all  
$s>1$, it holds that 
\begin{equation}\label{univariate_Legendre_nonweighted_samplingn}
g_
{{m}}^{\lin}( H(K_s), 
L_{\infty})
 \,\lesssim_s\, 
m^{-s+1}.
\end{equation}
\end{theorem}
 {
\begin{proof}
 The Christoffel function satisfies
$$
\Lambda_m^2 = 
{\Lambda(\Pi_m, L_{\infty})^2}
=
\sup_{x\in [-1,1]} \sum_{
{k=0}}^{
{m-1}
} |p_k (x) |^2 =
\sum_{k=0}^{m-1} (2k+1) = m^2\,.
$$
Taking into account that in this setting 
$\alpha(\Pi_m,H(K_s), L_2) = (1+ (m(m+1))^s )^{-1/2}$,
we get the result via Corollary~\ref{coro:general}.
\end{proof}
}

The estimate (\ref{univariate_Legendre_nonweighted_samplingn}) is better by a logarithmic factor than those obtained in Section 6.3 of \cite{PU}.

{
We now consider the sampling numbers of functions from $H(K_s)$ in the  {energy} space 
 {$\mathcal{H}^1$}, 
where 
 {$\mathcal{H}^1$}
denotes the space of all functions with weak first derivative being in $L_2$ and corresponding (semi-)norm.

\begin{theorem}
For all 
$s>2$, it holds that 

\[
g_m^{\lin}( H(K_s), 
 {\mathcal{H}^1})
\,\lesssim_s\, 
m^{-s+2}.
\]
\end{theorem}
\begin{proof}
It was shown by E. Schmidt in 1944 \cite[Formula (13)]{Schmidt1944}
 {that} for the Christoffel function of the set $V_m = \Pi_m$ of univariate polynomials of degree
 {at most $m-1$}
it
 {holds that}
$$
\Lambda(\Pi_m,
 {\mathcal{H}^1}
) \asymp m^2.
$$
Hence,  {Corollary~\ref{coro:general}  with  $\alpha(\Pi_m,H(K_s), L_2) = (1+ (m(m+1))^s )^{-1/2}$ yields the respective upper bound.}
\end{proof}

 {Note that}
there are many papers devoted to estimation of the recovery error in the 
 {$\mathcal{H}^1$}-norm {(see, e.g., \cite{Byrenheid_Dung_Sickel_Ullrich,Griebel_Knapek2000,Griebel_Knapek2009} and the references therein)}. This problem is connected to the numerical recovery of the solutions to a PDE (see, e.g., \cite{Binev_Bonito_2023}).
}

\vskip5mm

{\bf Acknowledgment.} The authors would like to thank Mathias Sonn\-leit\-ner for valuable comments on Section~\ref{subsec:mani}. 
The authors are grateful for the hospitality of the Leibniz Center Schloss Dagstuhl where this manuscript has been intensively discussed during the Dagstuhl Seminar 23351 {\it Algorithms and Complexity for Continuous Problems} (August/September 2023).

\bibliographystyle{amsplain}

\vskip2mm

\noindent\rule{\textwidth}{1pt}
\small

{\it 
Faculty of Computer Science and Mathematics, University of  Passau, Dr.-Hans-Kapfinger-Str. 30, 
94032 Passau, Germany. \\
Email address:} {\bf david.krieg@uni-passau.de}
\vskip1mm
{\it  Institute of Mathematics of NAS of Ukraine, Tereschenkivska Str.~3, 01024 Kyiv,
Ukraine; and Faculty of Mathematics,  Chemnitz University of Technology, Reichenhainer Str.~39, 
09126 Chemnitz, Germany. \\
Email address:}  {\bf pozharska.k@imath.kiev.ua}, \newline {\bf kateryna.pozharska@mathematik.tu-chemnitz.de}
\vskip1mm
{\it  Institute of Analysis and Department of Quantum Information and Computation,
Johannes Kepler University Linz, Altenberger Str. 69, 4040 Linz, Austria. \\
Email address:}  {\bf mario.ullrich@jku.at}
\vskip1mm
{\it  Faculty of Mathematics, Chemnitz University of Technology, Reichenhainer Str.
39, 09126 Chemnitz, Germany. \\
Email address:} {\bf tino.ullrich@mathematik.tu-chemnitz.de}

\end{document}